\documentclass[12pt]{article}
\usepackage{amsmath,amsfonts, amssymb,amsthm}
\bibstyle{plain}

\def\be{\begin{equation}}
\def\en{\end{equation}}
\def\bee{\begin{eqnarray*}}
\def\ene{\end{eqnarray*}}

\newtheorem{theorem}{Theorem}[section]
\newtheorem{corollary}[theorem]{Corollary}

\newtheorem{proposition}[theorem]{Proposition}

\theoremstyle{definition}

\newtheorem{remark}[theorem]{Remark}
\newtheorem{example}[theorem]{Example}

\numberwithin{equation}{section}

\usepackage{color}
\newcommand{\red}{\color{red}}
\newcommand{\black}{\color{black}}

\def\R{{\mathbb{R}}}

\def\P{{\mathbb{P}}}
\def\E{{\mathbb{E}}}

\vspace{1in}

\black

\title{\bf {Stein's method, logarithmic Sobolev\\ and transport inequalities}}

\author{Michel Ledoux\footnote{{Institut de Math\'ematiques de Toulouse, Universit\'e
de Toulouse, F-31062 Toulouse, France, and Institut Universitaire de France} {\tt   ledoux@math.univ-toulouse.fr}}\quad\quad\quad Ivan Nourdin\footnote{{Facult\'e des Sciences, de la Technologie et de la Communication; UR en Math\'ematiques. Luxembourg University, 6, rue Richard Coudenhove-Kalergi, L-1359 Luxembourg} {\tt ivan.nourdin@uni.lu}; IN was partially supported by the
ANR Grant ANR-10-BLAN-0121.}\quad  \quad\quad
Giovanni Peccati\footnote{{Facult\'e des Sciences, de la Technologie et de la Communication; UR en Math\'ematiques. Luxembourg University, 6, rue Richard Coudenhove-Kalergi, L-1359 Luxembourg} {\tt giovanni.peccati@gmail.com}; GP was partially supported by the Grant F1R-MTH-PUL-12PAMP  (PAMPAS) from Luxembourg University.}}

\begin{document}
\maketitle
\begin{abstract} {We develop connections between Stein's approximation
method, logarithmic Sobolev and transport inequalities by
introducing a new class of functional inequalities involving the relative entropy,
the Stein kernel, the relative Fisher information and the Wasserstein distance with respect to
a given reference distribution on $\R^d$. For the Gaussian model,
the results improve upon the classical logarithmic Sobolev inequality and
the Talagrand quadratic transportation cost
inequality. Further examples of illustrations include
multidimensional gamma distributions, beta distributions,
as well as families of log-concave densities.
As a by-product, the new inequalities are shown to be relevant
towards convergence to equilibrium, concentration inequalities and entropic convergence expressed
in terms of the Stein kernel. The tools rely on semigroup interpolation
and bounds, in particular by means of the iterated gradients of the Markov generator
with invariant measure the distribution under consideration.
In a second part, motivated by the
recent investigation by Nourdin, Peccati and Swan on Wiener chaoses, we address the issue
of entropic bounds on multidimensional functionals $F$ with the Stein kernel via
a set of data on $F$ and its gradients rather than on the Fisher information of the density.
A natural framework for this investigation is given by the Markov Triple structure
$(E, \mu, \Gamma)$ in which abstract Malliavin-type arguments may be developed
and extend the Wiener chaos setting. } \\

\end{abstract}

\bibliographystyle{plain}

\maketitle

\section{Introduction}

The classical logarithmic Sobolev inequality with respect to the standard Gaussian
measure $ d\gamma (x) =  (2\pi )^{-d/2} e^{-|x|^2/2} dx$ on $\R^d$ indicates that for
every probability $d\nu = h d\gamma$ with (smooth) density $h : \R^d \to \R_+$
with respect to $\gamma $,
\begin {equation} \label{e:logsob}
{\rm H} \big ( \nu \, | \, \gamma  \big )  \, = \,
    \int_{\R^d} h \log h \, d\gamma  \, \leq \,  \frac{1}{2} \int_{\R^d} \frac {|\nabla h|^2}{h} \, d\gamma
       \, = \, \frac12 \, {\rm I} \big ( \nu \, | \, \gamma  \big )
\end {equation}
where
$$
 {\rm H} \big ( \nu \, | \, \gamma  \big )  \, = \, \int_{\R^d} h \log h \, d\gamma \, = \,  {\rm Ent}_\gamma(h)
$$
is the relative entropy of $d \nu = h d\gamma $ with respect to $\gamma $ and
$$
 {\rm I} \big ( \nu \, | \, \gamma  \big )
 \, = \, \int_{\R^d} \frac {|\nabla h|^2}{h} \, d\gamma  \, = \, {\rm I}_\gamma(h)
$$
is the Fisher information of $\nu $ (or $h$) with respect to $\gamma$,
see e.g. \cite[Chapter II.5]{B-G-L} for a general discussion.
(Throughout this work, $| \cdot |$ denotes the Euclidean norm in $\R^d$.)

{Inspired by the recent investigation \cite{N-P-S}, this work
puts forward a new form of the logarithmic Sobolev inequality
\eqref {e:logsob} by considering a further ingredient, namely the
Stein discrepancy given by the Stein kernel of $\nu$.
A measurable matrix-valued map $\tau_\nu$ on $\R^d$
is said to be a {\it Stein kernel} for the (centered) probability $ \nu $
if for every smooth test function $\varphi  : \R^d \to \R$,
$$
\int_{\R^d} x\cdot \nabla\varphi  \, d\nu
\, = \, \int_{\R^d} { \big \langle \tau_\nu, {\rm Hess}(\varphi) \big  \rangle}_{\rm HS}  \, d\nu
$$
where ${\rm Hess} (\varphi) $ stands for the Hessian of $\varphi$, whereas
${\langle \cdot,\cdot\rangle}_{\rm HS}$ and ${\|\cdot\|}_{\rm HS}$ denote
the usual Hilbert-Schmidt scalar product and norm, respectively.
Note that while Stein kernels appear implicitly
in the literature about Stein's method
(see the original monograph \cite[Lecture VI]{Ste} of C. Stein,
as well as \cite{C2, C3, G-S1, G-S2}...), they gained momentum 
in recent years, specially
in connection with probabilistic approximations involving random variables
living on a Gaussian (Wiener) space
(see the recent monograph \cite{N-P-12} for an overview of this emerging area).
The terminology `kernel' with respect to `factor' seems the most appropriate
to avoid confusion with related but different existing notions.

According to the standard Gaussian integration
by parts formula from which $\tau_\gamma = {\rm Id}$,
the identity matrix in $\R^d$,
the proximity of $\tau_\nu $ with ${\rm Id}$ indicates that $\nu  $
should be close to the Gaussian distribution $\gamma $. Therefore,
whenever such a Stein kernel $\tau_\nu$ exists, the quantity, called {\it Stein discrepancy}
(of $\nu $ with respect to $\gamma$),
$$
{\rm S} \big (\nu\, | \, \gamma \big ) 
     \, = \, \bigg (\int_{\R^d} { \| \tau_\nu - {\rm Id}  \| }_{\rm HS} ^2 \, d \nu \bigg )^{1/2}
$$
becomes relevant as a measure of the proximity of $\nu$ and $\gamma$.
This quantity is actually at the root of the Stein method \cite {C-G-S,N-P-12}.
For example, in dimension one, the classical Stein bound
expresses that the total variation distance
${\rm TV}(\nu, \gamma)$ between a probability measure $\nu $ and the standard
Gaussian distribution $\gamma$ is bounded from above as
\begin {equation} \label {e:steintv}
{\rm TV}(\nu, \gamma) \, \leq  \, \sup \bigg | \int_\R \varphi '(x) d \nu (x)
      - \int _\R x \varphi (x) d\nu (x) \bigg |
\end {equation}
where the supremum runs over all continuously differentiable functions $\varphi  : \R \to \R$
such that ${\| \varphi  \|}_\infty \leq \sqrt { \frac {\pi}{2} } $ and ${\| \varphi ' \|}_\infty \leq 2$.
In particular, by definition of $\tau_\nu$
(and considering $\varphi'$ instead of $\varphi$),
$$
{\rm TV}(\nu, \gamma) \, \leq \, 2 \int_\R | \tau_\nu - 1 | d\nu
    \, \leq \, 2 \, {\rm S} \big (\nu\, | \, \gamma \big )
$$
justifying therefore the interest in the Stein discrepancy  (see also \cite{C-P-U}).
It is actually a main challenge addressed in \cite {N-P-S} and this work to
investigate the multidimensional setting in which inequalities such as \eqref {e:steintv} are no more available.

With the Stein discrepancy ${\rm S} (\nu\, | \, \gamma)$,
we emphasize here the inequality, for every probability $d\nu = h d\gamma$,
\begin{equation}\label{e:hsi1}
{\rm H} \big ( \nu \, | \, \gamma  \big )  \,  \leq \,  \frac{1}{2} \,  {\rm S}^2 \big ( \nu \, | \, \gamma  \big )
      \log \bigg ( 1 + \frac{ {\rm I}  ( \nu \, | \, \gamma   )}
      {{\rm S}^2 ( \nu \, | \, \gamma   ) }  \bigg )
\end{equation}
as a new improved form of the logarithmic Sobolev inequality \eqref {e:logsob}.
In addition, this inequality \eqref {e:hsi1} transforms bounds on the Stein
discrepancy into entropic bounds, hence allowing for entropic approximations
(under finiteness of the Fisher information). 
Indeed as is classical, the relative entropy
$ {\rm H} (\nu \, | \, \gamma  ) $ is another measure of the proximity
between two probabilities $\nu$ and $\gamma$ (note that $ {\rm H} (\nu \, | \, \gamma  )  \geq 0$ and
$ {\rm H} (\nu \, | \, \gamma  ) =0$ if and only if $\nu = \gamma$),
which is moreover stronger than the total variation distance by
the Pinsker-Csizs\' ar-Kullback inequality
$${\rm TV}(\nu , \gamma ) \, \leq \, 
     \sqrt { \frac{1}{2} \, {\rm H} \big (\nu \, | \, \gamma \big ) } $$
 (see, e.g.~\cite[Remark 22.12]{V}). 

The proof of \eqref {e:hsi1} is achieved by the classical
interpolation scheme along the Ornstein-Uhlenbeck semigroup ${(P_t)}_{t \geq 0}$
towards the logarithmic Sobolev inequality, but modified
for time $t$ away from $0$ by a further integration by parts involving the Stein kernel.
Indeed, while the exponential decay
${\rm I}_\gamma (P_t h) \leq e^{-2t} \, {\rm I}_\gamma (h)$
of the Fisher information
classically produces the logarithmic Sobolev inequality \eqref {e:logsob},
the argument is supplemented by a different control of
${\rm I}_\gamma (P_t h)$ by the Stein discrepancy for $t>0$.

We call the inequality \eqref {e:hsi1} HSI, connecting entropy H, Stein discrepancy S and
Fisher information I, by analogy with the celebrated Otto-Villani HWI inequality \cite{O-V} relating
entropy H, (quadratic) Wasserstein distance W
(${\rm W}_2$) and Fisher information I. We actually provide in Section~\ref {S3}
a comparison between the HWI and HSI inequalities (suggesting even an HWSI inequality).
Moreover, based on the approach developed in \cite {O-V}, we prove that
\begin{equation}\label{e:HWS}
{\rm W}_2(\nu, \gamma) \, \leq \,  {\rm S} \big ( \nu \, | \, \gamma  \big )
    \arccos\Big(e^{-\frac{ {\rm H}(\nu\, |\, \gamma)}{{\rm S}^2(\nu\, |\, \gamma)}}\Big),
\end{equation}
an inequality that improves upon the celebrated Talagrand quadratic transportation cost inequality \cite{T}
$$
{\rm W}_2^2 (\nu , \gamma ) \, \leq 2 \,  {\rm H}  \big ( \nu \, | \, \gamma  \big )
$$
(since $\arccos(e^{-r}) \leq \sqrt{2r}$ for every $r\geq 0$).
We shall refer to \eqref{e:HWS} as the `WSH inequality'.
Note also that $ {\rm W}_2(\nu, \gamma) \, \leq \,  {\rm S} ( \nu \, | \, \gamma  )$
so that, as entropy, the Stein discrepancy is a stronger measurement than the
Wasserstein metric ${\rm W}_2$.

The new HSI inequality put forward in this work has a number of significant
applications to exponential convergence to equilibrium and concentration inequalities.
For example, the standard exponential decay of entropy
$ {\rm H} ( \nu^t \, | \, \gamma   ) \leq  e^{-2t} \, {\rm H} ( \nu^0 \, | \, \gamma   )$
along the flow $d\nu^t = P_t h d\gamma$, $ t \geq 0$ ($\nu^0 = \nu$, $\nu^\infty = \gamma$),
which characterizes the logarithmic Sobolev inequality \eqref {e:logsob}
may be strengthened under finiteness of the Stein
discrepancy $ {\rm S} = {\rm S} (\nu \, | \, \gamma ) = {\rm S} (\nu^0 \, | \, \gamma )$ into
\begin {equation} \label {e:decay2}
 {\rm H} \big (\nu^t \, | \, \gamma \big ) \, \leq \, 
    \frac {e^{-4t}} {e^{-2t} + \frac {1 - e^{-2t}}{{\rm S}^2} \, {\rm H} (\nu^0 \, | \, \gamma ) } 
     \, {\rm H} \big (\nu^0 \, | \, \gamma \big)
     \, \leq \, \frac {e^{-4t}}{1 - e^{-2t}} \, {\rm S}^2 \big (\nu^0 \, | \, \gamma \big ) 
\end {equation}
(see Corollary \ref{cor.decay} for a precise statement).
On the other hand, logarithmic Sobolev inequalities are classically
related to (Gaussian) concentration inequalities by means of the Herbst argument
(cf.~e.g.~\cite {L2,B-L-M}). Stein's method has also been used to this task in
\cite {C3}, going back however to the root of the methodology of exchangeable pairs.
The basic principle emphasized in this work actually allows us to directly quantify
concentration properties of a probability $\nu $ on $\R^d$ in terms of its Stein discrepancy
with respect to the standard Gaussian measure.
As a result, for any $1$-Lipschitz function $ u : \R^d \to \R$ with mean zero,
and any $ p \geq 2$,
\begin {equation} \label {eq.moments2}
\bigg (\int_{\R^d} |u|^p d \nu \bigg)^{1/p} \, \leq \,  C  \Big (  {\rm S}_p \big ( \nu \, | \, \gamma \big ) 
   + {\sqrt p} + {\sqrt p} \, {\sqrt { {\rm S}_p \big ( \nu \, | \, \gamma \big ) }} \, \Big)
\end {equation}
where $ C>0$ is numerical and
$$
{\rm S}_p \big ( \nu \, | \, \gamma \big ) \, = \,
      \bigg (\int_{\R^d} {\| \tau _\nu - {\rm Id}\|}_{\rm HS}^p \, d\nu \bigg)^{1/p}.
$$
(When $\nu = \gamma$, the result fits the standard Gaussian concentration
properties.)
In other words, the growth of the Stein discrepancy ${\rm S}_p  ( \nu \, | \, \gamma )$
in $p$ entails concentration properties of the measure $\nu$ in terms of the growth of its moments.
This result is one very first instance showing how to directly transfer informations
on the Stein kernel into concentration properties. It yields for example that if
$T_n = \frac {1}{\sqrt n} (X_1 + \cdots + X_n)$ where $X_1, \ldots , X_n$
are independent with common distribution $\nu$ in $\R^d$ with mean zero and covariance
matrix ${\rm Id}$, for any $1$-Lipschitz function $u : \R^d \to \R$ such that 
$\E(u(T_n)) = 0$,
$$
\P \big ( u (T_n) \geq r \big ) \, \leq \, C \, e^{-r^2 /C}
$$
for all $0 \leq r \leq r_n$ where $r_n \to \infty$
according to the growth of ${\rm S}_p(\nu \, | \, \gamma)$ as $p\to \infty$. 

While put forward for the Gaussian measure $\gamma$, the question
of the validity of (a form of) the HSI and WSH inequalities for other reference measures
should be addressed. Natural examples exhibiting HSI inequalities
may be described as invariant measures of second order differential
operators (on $\R^d$) in order to run the semigroup interpolation scheme.
The prototypical example is of course
the Ornstein-Uhlenbeck operator with the standard Gaussian measure as invariant
measure. But gamma or beta distributions associated to Laguerre or Jacobi operators
may be covered in the same way, as well as families of log-concave measures.
It should be mentioned that the definition of Stein kernel has then
to be adapted to the diffusion coefficient of the underlying differential operator.
The use of second order differential operators in order to study multidimensional probabilistic approximations plays a fundamental role in the so-called {\it generator approach} to Stein's method, as introduced in the seminal references \cite{Ba, G}; see also \cite{R}
for a survey on the subject.
A convenient setting to work out this investigation is the one of Markov
Triples $(E, \mu, \Gamma)$ and semigroups ${(P_t)}_{t \geq 0}$
as emphasized in \cite{B-G-L} allowing for the $\Gamma$-calculus
and the necessary heat kernel bounds in terms of the iterated gradients $\Gamma_n$.
In particular, while the classical Bakry-\'Emery $\Gamma_2$ criterion
\cite {B-E,B-G-L} ensures the validity
of the logarithmic Sobolev inequality in this context, it is
worth mentioning that the analysis towards the HSI bound makes critical use of the
associated $\Gamma_3$ operator, a rather new feature in the study
of functional inequalities. 
\medskip

As alluded to above, the HSI inequality \eqref{e:hsi1}
is designed to yield entropic central limit theorems for
sequences of probability measures of the form $d\nu_n = h_nd\gamma$, $n\geq 1$,
such that $s_n ={\rm S} (\nu_n \, | \, \gamma ) \to 0$ and
$$
\log \bigg ( 1 + \frac{ {\rm I} ( \nu_n \, | \, \gamma  )}{s_n^2}  \bigg ) \, = \,  o(s^{-2}_n),
\quad n\to\infty.
$$
This is achieved, for instance, when the sequence
$ {\rm I}  ( \nu_n \, | \, \gamma   )$, $n\geq 1$, is bounded.
However, the principle behind the HSI inequality may actually be used to deduce entropic
convergence (with explicit rates) in more delicate situations,
including cases for which $ {\rm I}  ( \nu_n \, | \, \gamma  )\to \infty$.
Indeed, it was one main achievement
of the work \cite{N-P-S} in the context of Wiener chaoses to set up bounds involving
entropy and the Stein discrepancy without conditions on the Fisher information.
Specifically, it was proved in \cite {N-P-S} that the entropy with respect
to the Gaussian measure $\gamma$ of the distribution on $\R^d$ of a vector
$F= (F_1, \ldots , F_d)$ of Wiener chaoses may be controlled by the Stein discrepancy,
providing the first multidimensional entropic approximation results in this context.
The key feature underlying the HSI inequality is the control as
$t \to 0$ of the Fisher information ${\rm I}_\gamma (P_th)$ along the semigroup
(where $h$ the density with respect to $\gamma$ of the law of~$F$)
by the Stein discrepancy.
The arguments in \cite{N-P-S} actually provide the suitable small time
behavior of ${\rm I}_\gamma (P_th)$ relying on specific properties of the functionals (Wiener chaoses)
under investigation and tools from Malliavin calculus.

In the second part of the work, we therefore develop a general approach to cover the results
of \cite {N-P-S} and to include a number of further potential instances of interest.
As before, the setting of a Markov Triple $(E, \mu, \Gamma)$
provides a convenient abstract framework to achieve this goal in which
the $\Gamma$-calculus appears as a kind of substitute
to the Malliavin calculus in this context.
Let $\Psi $ be the function $ 1 + \log r$ on $\R_+$ but linearized by $r$ on $[0,1]$, that is,
$\Psi (r) = 1 + \log r $ if $r \geq 1$ and $ \Psi (r) = r $ if $0 \leq r \leq 1$
(note that $\Psi (r ) \leq r$ for every $ r \in \R_+$).
A typical conclusion is a bound of the type
\begin {equation} \label {e:cf}
  {\rm H} \big ( \nu_F \, | \, \gamma \big)
       \, \leq \,   C_F \, {\rm S}^2 \big (\nu_F \, | \, \gamma \big ) \,
       \Psi \bigg ( \frac{  \widetilde{C}_F } { {\rm S}^2(\nu_F \, | \, \gamma)}\bigg) 
\end {equation}
of the relative entropy of the distribution $\nu_F$ of a vector
$F = (F_1, \ldots ,F_d)$ on $(E, \mu, \Gamma)$ with respect to $\gamma$
by the Stein discrepancy $ {\rm S}(\nu_F \, | \, \gamma)$, where $C_F, \widetilde{C}_F > 0$ depend
on integrability properties of $F$, the carr\'e du champ operators $\Gamma (F_i, F_j)$,
$i,j = 1, \ldots, d$, and the inverse of the determinant of the matrix
${(\Gamma (F_i, F_j))}_{1 \leq i, j \leq d}$. In particular,
$ {\rm H}(\nu_F \, | \, \gamma) \to~0$ as $ {\rm S}(\nu_F \, | \, \gamma) \to 0$
providing therefore entropic convergence under the Stein discrepancy.
The general results obtained here cover not only normal approximation but also
gamma approximation.

The inequality \eqref {e:cf} thus transfers bounds on the Stein discrepancy to entropic bounds.
The issue of controlling the Stein discrepancy $ {\rm S}(\nu_F \, | \, \gamma)$ itself
(in terms of moment conditions for example)
is not addressed here, and has been the subject of numerous recent studies
around the so-called Nualart-Peccati fourth moment theorem (cf.~\cite {N-P-12}).
This investigation is in particular well adapted to functionals $ F= (F_1, \ldots , F_d)$
whose coordinates are eigenfunctions of the underlying Markov generator.
See \cite{A-C-P, A-M-P, L3} for several results in this direction and
\cite[Chapters 5-6]{N-P-12} for a detailed discussion of
estimates on ${\rm S}(\nu_F \, | \, \gamma)$ that are available for random
vectors $F$ living on the Wiener space.

\medskip

The structure of the paper thus consists of two main parts, the first one devoted to the
new HSI and WSH
inequalities, the second one to an investigation of entropic bounds via the Stein discrepancy.
Section~\ref{S2} is devoted to the proof and discussions of
the HSI inequality in the Gaussian case, with a first sample of illustrations
and applications to convergence to equilibrium and measure concentration.
In Section~\ref{S3}, we investigate connections between the Stein discrepancy, Wasserstein
distances and transportation cost inequalities, in particular the HWI inequality, and
establish the WSH inequality.
Extensions of the HSI inequality to more general distributions
arising as invariant probability measures of second order differential operators
are addressed in Section~\ref{S4}. The second part consists of Section~\ref {S5} which develops
a general methodology (in the context of Markov Triples) to reach entropic bounds
on densities of families of functionals under conditions which do not necessarily
involve the Fisher information.

\section {Logarithmic Sobolev inequality and Stein discrepancy} \label{S2}

Throughout this section, we fix an integer $d\geq 1$ and let
$\gamma = \gamma^d$ indicate the standard Gaussian measure on the Borel sets of $\R^d$.

\subsection{Stein kernel and discrepancy} \label {S2.1}

Let $\nu$ be a probability measure on the Borel sets of $\R^d$.
In view of
the forthcoming definitions, we shall always assume (without loss of generality)
that $\nu$ is centered, that is,
$ \int_{\R^d} x_j \, d\nu(x)=0$, $j=1, \ldots, d$.

As alluded to in the introduction, a measurable matrix-valued map on $\R^d$
$$
x \, \mapsto \,  \tau_\nu(x) \, = \,  \big \{\tau_\nu ^{ij}(x) : i,j=1, \ldots,d \big \}
$$
is said to be a Stein kernel for $\nu $ if
$\tau_\nu^{ij} \in {\rm L}^1(\nu)$ for every $i,j$ and, for every smooth $\varphi  : \R^d \to \R$,
\begin{equation}\label{e:steinmatrix}
\int_{\R^d} x\cdot \nabla \varphi \, d\nu
   \, = \, \int_{\R^d} { \big \langle \tau_\nu , {\rm Hess}( \varphi) \big \rangle}_{\rm HS} \, \,d\nu.
\end{equation}
Observe from \eqref{e:steinmatrix} that, without loss of generality, one
may and will assume in the sequel that $\tau_\nu^{ij}(x)=\tau_\nu^{ji}(x)$ $\nu$-a.e.,
$i,j=1,\ldots,d$.
Also, by choosing $ \varphi  =x_i$, $i=1,\ldots,d$, in  \eqref{e:steinmatrix} one sees that,
if $\nu$ admits a Stein kernel, then $\nu$ is necessarily centered.
Moreover, by selecting $ \varphi = x_ix_j$, $i,j=1,\ldots ,d$,
and since $\tau^{ij}_\nu=\tau^{ji}_\nu$,
$$
 \int_{\R^d} x_ix_j \, d\nu \, = \, \int_{\R^d} \tau_\nu^{ij} d\nu,  \quad i,j=1, \ldots ,d
$$
(and in particular $\nu$ has finite second moments).

\begin{remark} \label {r:kernel}

\begin{itemize}

\item [(a)] Let $d=1$ and assume that $\nu$ has a density $\rho $
with respect to the Lebesgue measure
on $\R$. In this case, it is easily seen that, whenever it exists, the Stein kernel $\tau_\nu$
is uniquely determined (up to sets of zero Lebesgue measure). Moreover, under standard
regularity assumptions on $\rho $, one deduces from integration by parts that a
version of $\tau_\nu$ is given by
\begin {equation} \label{e:kerneldimensionone}
\tau_\nu(x) \, = \,  \frac {1}{\rho (x)}\int_x^\infty y \rho (y)dy
\end {equation}
for $x$ inside the support of $\rho$.

\item [(b)] In dimension $d\geq 2$, a Stein kernel $\tau_\nu$ may not be
unique -- see \cite[Appendix~A]{N-P-S2}.

\item [(c)] It is important to notice that, in dimension $d\geq 2$, the
definition \eqref{e:steinmatrix} of Stein
kernel is actually weaker than the one used in \cite{N-P-S, N-P-S2}. Indeed,
in those references a Stein kernel $\tau_\nu$ is required to satisfy the
stronger `vector' (as opposed to the trace identity \eqref{e:steinmatrix}) relation
\begin{equation}\label{e:strongersm}
\int_{\R^{d}} x  \, \varphi \, d\nu  \, = \,  \int_{\R^d} \tau_\nu \nabla \varphi \,  d\nu
\end{equation}
for every smooth test function $\varphi : \R^d \to \R$. The definition \eqref{e:steinmatrix}
of a Stein kernel adopted
in the present paper allows one to establish more transparent connections between
normal and non-normal approximations, such as the ones explored in Section 4.
Observe that it will be nevertheless necessary to use Stein kernels in the strong sense
\eqref {e:strongersm} when dealing with Wasserstein distances of
order $\neq$ 2 in Section~\ref{ss:W3}.
\end{itemize}
\end{remark}

Definition \eqref{e:steinmatrix} is directly inspired
by the Gaussian integration by parts formula according to which
\begin{equation}\label{e:gauss}
\int_{\R^d} x\cdot \nabla \varphi   \, d\gamma
    \ = \,  \int_{\R^d}  \Delta \varphi  \, d\gamma
  \, = \, \int_{\R^d} {\big \langle {\rm Id},{\rm Hess}(\varphi) \big \rangle}_{\rm HS} \, d\nu
\end{equation}
so that the proximity of $\tau_\nu $ with the identity matrix
${\rm Id}$ indicates that $\nu $ should be close to $\gamma $.
In particular, it should be clear that the notion of Stein kernel in the sense
of \eqref{e:steinmatrix} is motivated by normal approximation.
Section~\ref{S4} will introduce analogous definitions adapted to the target measure
in the context of the generator approach to Stein's method.
Whenever a Stein kernel exists,
we consider to this task the quantity, called Stein discrepancy
of $\nu$ with respect to $\gamma$ in the introduction,
$$
{\rm S} \big (\nu\, | \, \gamma \big ) \, = \,   {\| \tau_\nu - {\rm Id} \|}_{2,\nu } 
     \, = \,  \bigg (\int_{\R^d} { \| \tau_\nu - {\rm Id}  \| }_{\rm HS} ^2 \, d\nu \bigg )^{1/2}.
$$
(Note that ${\rm S} (\nu\, | \, \gamma ) $ may be infinite if one
of the $\tau_\nu^{ij}$'s is not in ${\rm L}^2(\nu)$.)
Whenever ${{\rm S} (\nu\, | \, \gamma ) = 0}$, then $\nu = \gamma$
since $\tau_\nu$ is the identity matrix (see e.g.~\cite[Lemma~4.1.3]{N-P-12}).
Observe also that if $C $ denotes the covariance matrix of $\nu$, then
\begin{equation}\label{e:s2bound}
{\rm S}^2 \big (\nu\, | \, \gamma\big)
  \, = \,   \sum_{i,j = 1}^d {\rm Var}_\nu\, (\tau_\nu^{ij})  + {\| C- {\rm Id}\|}_{\rm HS}^2,
\end{equation}
where ${\rm Var}_\nu$ indicates the variance under the probability
measure $\nu$.

\subsection{The Gaussian HSI inequality}\label{S21}

As before, write $d\nu = hd\gamma$ to indicate a centered probability measure on $\R^d$
which is absolutely continuous with density $h$
with respect to the standard Gaussian distribution $\gamma$.
We assume that there exists a Stein kernel $\tau_\nu$ for $\nu $
as defined in \eqref{e:steinmatrix} of the preceding section.

The following result emphasizes the Gaussian HSI inequality
connecting entropy ${\rm H}$,
Stein discrepancy ${\rm S}$ and Fisher information ${\rm I}$.
In the statement, we use the conventions $0\log (1+ \frac {s}{0}) = 0$
and $\infty\log (1+ \frac {s}{\infty}) = s$ for every $s\in [0,\infty]$,
and $r \log (1+ \frac {\infty}{r}) = \infty$ for every $r\in (0,\infty)$.

\begin{theorem}[Gaussian HSI inequality]\label{t:hsi}
For any centered probability measure ${d\nu = hd\gamma}$ on $\R^d$ with
smooth density $h$ with respect to $\gamma$,
\begin{equation}\label{e:hsi2}
{\rm H} \big ( \nu \, | \, \gamma  \big )  \,   \leq \,  \frac{1}{2} \,  {\rm S}^2 \big ( \nu \, | \, \gamma  \big )
      \log \bigg ( 1 + \frac{ {\rm I} ( \nu \, | \, \gamma  )}{{\rm S}^2 ( \nu \, | \, \gamma  ) }  \bigg ).
\end{equation}
\end {theorem}

Since $ r \log \big ( 1+ \frac {s}{r} \big ) \leq s$ for every $r>0$, $ s\geq 0$,
the HSI inequality \eqref {e:hsi2} improves upon the standard
logarithmic Sobolev inequality \eqref{e:logsob}.
It may be observed also that the HSI inequality immediately produces the
(classical) equality case in
this logarithmic Sobolev inequality. Indeed, due to the centering hypothesis,
equality is achieved only for the Gaussian measure $\gamma $ itself (if not, recenter
first $\nu$ so that the only extremals of \eqref{e:logsob}
have densities $e^{m\cdot x - |m|^2/2}$, $ m \in \R^d$,
with respect to $\gamma$). To this task, assume by contradiction that $ {\rm S} ( \nu \, | \, \gamma) >0$.
Then, if 
$ {\rm H} ( \nu \, | \, \gamma) = \frac {1}{2} \, {\rm I} ( \nu \, | \, \gamma) $,
the HSI inequality \eqref{e:hsi2} yields
$$
\frac {{\rm I} ( \nu \, | \, \gamma)}{{\rm S}^2 ( \nu \, | \, \gamma)}
   \, \leq \, \log \bigg ( 1 + \frac {{\rm I} ( \nu \, | \, \gamma)}{{\rm S}^2 ( \nu \, | \, \gamma)} \bigg)
$$
from which $ {\rm I} ( \nu \, | \, \gamma) = 0$, and therefore $ \nu = \gamma$,
which is in contrast with the assumption $ {\rm S} ( \nu \, | \, \gamma) >0$.
As a consequence, we infer that $ {\rm S} ( \nu \, | \, \gamma) =0$, from which it follows
that $\nu=\gamma$}.

\medskip

The HSI inequality \eqref{e:hsi2} may be extended to the case of a
centered Gaussian distribution on $\R^d$ with a general non-degenerate covariance matrix $C$.
We denote such a measure by $\gamma_C$, so that
$\gamma = \gamma_{\rm Id}$. We also denote by ${\|C\|}_{\rm op}$
the operator norm of $C$, that is, ${\|C\|}_{\rm op}$ is the largest eigenvalue of $C$.

\begin{corollary}[Gaussian HSI inequality, general covariance] \label {c:hsi}
{Let $\gamma_C$ be as above (with $C$ non-singular), and let $d\nu = h d\gamma_C$ 
be centered with smooth probability density $h$ with respect to $\gamma_C$. Assume that
$\nu$ admits a Stein kernel $\tau_\nu$ in the sense of \eqref {e:steinmatrix}. Then,}
$$
{\rm H} \big (\nu \, | \, \gamma_C \big)  \, \leq \, \frac 12 \,
{\big \| C^{-\frac12} \tau_\nu\,C^{-\frac12} - {\rm Id} \big \|}^2_{2,\nu}
\log\bigg ( 1+ \frac{{\|C\|}_{\rm op}\,\,{\rm I}(\nu\, |\,  \gamma_C)}
    {\| C^{-\frac12} \tau_\nu \,C^{-\frac12}- {\rm Id}\|^2_{2,\nu}}\bigg ),
$$
where $C^{-\frac12}$ denotes the unique symmetric non-singular matrix
such that ${(C^{-\frac12})}^2=C^{-1}$.
\end{corollary}

Corollary~\ref {c:hsi} is easily deduced from Theorem~\ref{t:hsi}
and details are left to the reader. The argument simply uses that if $M$ is the
unique non-singular symmetric matrix such that $C=M^2$, then
$ {\rm H} (\nu \, |\, \gamma_C ) = {\rm H} (\nu^0 \, |\, \gamma ) $
where $d\nu^0(x) = h(Mx) d\gamma(x)$.

\subsection{Proof of the Gaussian HSI inequality}\label{S22}

According to our conventions, if either ${\rm S}(\nu\, | \, \gamma)$
or ${\rm I}(\nu\, | \, \gamma)$ is infinite, then \eqref{e:hsi2}
coincides with the logarithmic Sobolev inequality \eqref{e:logsob}.
On the other hand, if  ${\rm S}(\nu\, | \, \gamma) $ or ${\rm I}(\nu\, | \, \gamma)
$ equals zero, then $\nu = \gamma$, and therefore ${\rm H}(\nu\, | \, \gamma) = 0$.
It follows that, in order to prove \eqref{e:hsi2}, we can assume without loss of
generality that ${\rm S}(\nu\, | \, \gamma)$
and ${\rm I}(\nu\, | \, \gamma)$ are both non-zero and finite.

\smallskip

The proof of Theorem \ref{t:hsi} is based on the heat flow interpolation
along the Ornstein-Uhlenbeck semigroup. We recall a few basic facts in this regard,
and refer the reader to e.g. \cite[Section~2.7.1]{B-G-L} for any unexplained
definition or result. Let thus ${(P_t)}_{t \geq 0}$ be the Ornstein-Uhlenbeck semigroup on
$\R^d$ with infinitesimal generator
\begin {equation} \label{e:ou}
{\cal L} f \, = \,  \Delta f  - x \cdot \nabla f
   \, = \,  \sum_{i =1}^d  \frac{\partial^2 f}{\partial x_i^2}
     - \sum_{i=1}^d x_i \, \frac{\partial f}{\partial x_i}
\end {equation}
(acting on smooth functions $f$), invariant and symmetric with respect to $\gamma $.
We shall often use the fact that the action of $P_t$ on smooth
functions $f : \R^d\to \R$ admits the integral representation
(sometimes called {\it Mehler's formula})
$$ P_t f(x) \, = \,  \int_{\R^d} f \big ( e^{-t} x + \sqrt {1- e^{-2t}} \, y \big ) d\gamma (y),
    \quad t \geq 0, \, \, x \in \R.  $$
The semigroup is trivially extended to vector-valued functions $f : \R^d \to \R^d$. In
particular, if $f : \R^d \to \R$ is smooth enough,
\begin {equation} \label {eq1.3}
 \nabla P_t f \, = \,  e^{-t} P_t (\nabla f) .
\end {equation}
{One technical important property (part of the much more general Bismut formulas
in a geometric context \cite{Bi,B-G-L}) is the identity, between vectors in $\R^d$,
\begin {equation} \label {eq1.4}
 P_t (\nabla f)(x) \, = \,  \frac{1}{\sqrt {1- e^{-2t}}}
     \int_{\R^d} y \,  f \big ( e^{-t} x + \sqrt {1- e^{-2t}} \,  y \big ) d\gamma (y),
\end {equation}
owing to a standard integration by parts of the Gaussian density.}

The generator $\cal L$ is a diffusion and satisfies the integration by parts formula
\begin {equation} \label {e:ipp}
 \int_{\R^d} f \, {\cal L} g \,  d\gamma  \, = \, -  \int_{\R^d} \nabla f \cdot \nabla g \, d\gamma
\end {equation}
on smooth functions $f, g : \R^d \to \R$.
In particular, given the smooth probability density $h $
with respect to $\gamma $,
$$
  {\rm I}_\gamma  (h)  \, = \, \int_{\R^d} \frac {|\nabla h|^2}{h} \, d\gamma
    \, = \, \int_{\R^d} |\nabla (\log h) |^2 h  d\gamma 
       \, = \,  - \int_{\R^d} {\cal L}(\log h) h  d\gamma    .
$$
As $d \nu = h d \gamma$, setting $ v = \log h$,
\begin {equation} \label {eq1.5}
 {\rm I} \big (\nu \, | \, \gamma \big ) \, = \,
 {\rm I}_\gamma  (h) \, = \,    \int_{\R^d} |\nabla v|^2 d\nu 
       \, = \,  - \int_{\R^d} {\cal L}v \, d\nu  .
\end {equation}
(These expressions should actually be considered for $h + \varepsilon$ as $\varepsilon \to 0$.)
Using $P_t h$ instead of $h$ in the previous relations and
writing $v_t = \log P_t h$, one deduces from the symmetry of $P_t$ that
\begin {equation} \label {eq1.6}
{\rm I} \big (\nu^t \, | \, \gamma \big ) \, = \,
 {\rm I}_\gamma  (P_t h) \,  = \,   \int_{\R^d} \frac {|\nabla P_t h|^2}{ P_t h} \,  d\gamma
  \,  = \,   - \int_{\R^d} {\cal L} v_t  \, P_t h d\gamma
  \,  = \,   - \int_{\R^d} {\cal L} P_t v_t  \, d\nu   . 
\end {equation}
Recall finally that if $d\nu^t = P_t h d \gamma$, $t \geq 0$ (with $\nu^0 = \nu $ and
$\nu^t \to \gamma$), the classical {\it de~Bruijn's formula}
(see e.g.~\cite[Proposition~5.2.2]{B-G-L}) indicates that
\begin {equation} \label {e:bruijn}
\frac {d}{dt} \, {\rm H} \big ( \nu^t \, | \, \gamma \big) = - \, {\rm I} \big ( \nu^t \, | \, \gamma \big).
\end {equation}

Theorem \ref{t:hsi} will follow from the next Proposition \ref{p:up}.
In this proposition, {\it (i)} corresponds to the integral version of \eqref {e:bruijn}
whereas {\it (ii)} describes the well-known exponential decay of the Fisher
information along the Ornstein-Uhlenbeck semigroup. This decay actually yields
the logarithmic Sobolev inequality \eqref {e:logsob}, see \cite[Section~5.7]{B-G-L}.
The new third point {\it (iii)} is a reformulation of
\cite[Theorem~2.1]{N-P-S} for which we provide a self-contained proof.
It describes an alternate bound on the Fisher information
along the semigroup in terms of the Stein discrepancy for values of $t>0$ away from $0$.
It is the combination of {\it (ii)} and {\it (iii)} which will produce
the HSI inequality.
Point $(iv)$ will be needed in the forthcoming proof of the WSH inequality \eqref{e:HWS},
as well as in the proof of Proposition \ref{p:ws}
providing a direct bound of the Wasserstein distance ${\rm W}_2$ by the Stein discrepancy.

\begin{proposition} \label{p:up}  Under the above notation and assumptions,
denote by $\tau_\nu$ a Stein kernel of $d\nu = h d\gamma$.
For every $t>0$, recall $d\nu^t = P_th \,d\gamma$, and write $v_t = \log P_th$. Then,
\begin{itemize}
\item[(i)]{\rm (Integrated de Bruijn's formula)}
\begin {equation} \label {eq1.7}
{\rm H} \big (\nu\, |\, \gamma \big ) \, = \,    {\rm Ent}_\gamma (h)
     \, = \,  \int_0^\infty {\rm I}_\gamma (P_t h) dt .
\end {equation}

\item[ (ii)] {\rm (Exponential decay of Fisher information)}
For every $t \geq 0$,
\begin{equation}\label{e:decay}
 {\rm I} \big (\nu^t \, | \, \gamma \big ) \, = \,
 {\rm I}_\gamma  (P_t h) \, \leq \,  e^{-2t} \,  {\rm I}_\gamma  ( h)
   \, = \, e^{-2t} \, {\rm I} \big (\nu^0 \, | \, \gamma \big ).
\end{equation}

\item[ (iii)] For every $t>0$,
\begin {equation}   \label{e:nps}
  {\rm I}_\gamma(P_t  h)
 \, = \,  \frac{e^{-2t}}{\sqrt { 1 - e^{-2t}}}\!\!
   \int_{\R^d} \!\int_{\R^d}\! \Big [ \big (\tau_\nu (x) - {\rm Id} \big ) y
         \cdot \nabla  v_t   \big (e^{-t } x + \sqrt {1 - e^{-2t}} \, y \big ) \Big ]
            d \nu (x) d \gamma (y).
\end {equation}
As a consequence, for every $t>0$,
 \begin{equation} \label {eq1.9}
{\rm I} \big (\nu^t \, | \, \gamma \big ) \, = \,  {\rm I}_\gamma  (P_t h) 
     \, \leq \,   \frac{e^{-4t}}{1 - e^{-2t}} \, {\| \tau_\nu - {\rm Id} \|}_{2,\nu } 
       \, = \,  \frac{e^{-4t}}{1 - e^{-2t}} \, {\rm S}^2 \big (\nu^0 \, | \, \gamma \big ) . 
\end{equation}

\item[(iv)]
 {\rm (Exponential decay of Stein discrepancy)} For every $t \geq 0$,
\begin{equation}\label{e:decaystein}
{\rm S} \big (\nu^t \, | \, \gamma \big ) \, \leq \,  e^{-2t} \, {\rm S} \big (\nu^0 \, | \, \gamma \big ).
\end{equation}
%

\end{itemize}
\end{proposition}

\begin{proof}
In view of the preceding discussion, only the proofs of {\it (iii)} and $(iv)$ need to be detailed.
Throughout the various analytical arguments below, it may be assumed that
the density $h$ is regular enough, the final conclusions being then reached by approximation
arguments as e.g. in \cite {O-V, B-G-L}.
Starting with {\it (iii)}, use \eqref {eq1.6} and the
definition \eqref{e:steinmatrix} of $\tau_\nu $ to write, for any $ t >0$,
\begin{equation} \begin {split} \label {eq1.8}
  {\rm I}_\gamma  (P_t h) \, = \,  - \int_{\R^d} {\cal L} P_t v_t  \, d\nu
    &  \, = \,  - \int_{\R^d} \big [ \Delta P_t v_t  - x \cdot \nabla P_t v_t \big ] d\nu \\
   & \,  = \,   \int_{\R^d} {\big \langle \tau_\nu - {\rm Id} ,{\rm Hess}(P_t v_t) \big \rangle}_{\rm HS} \,
      d\nu.  \\
\end {split} \end {equation}
Now, for all $i, j = 1, \ldots , d$, by \eqref {eq1.3} and \eqref {eq1.4},
$$ \partial _{ij} P_t v_t (x) = e^{-2t} P_t (\partial _{ij} v_t) (x)
     \, = \,  \frac{e^{-2t}}{\sqrt { 1 - e^{-2t}}} \int_{\R^d}
           y_i \, \frac{\partial v_t}{\partial x_j} \big (e^{-t } x + \sqrt {1 - e^{-2t}} \, y \big ) d\gamma (y).$$
Hence
\begin{equation*} \begin {split}
   \int_{\R^d}  &{\big  \langle  \tau_\nu  -  {\rm Id} ,{\rm Hess}(P_t v_t) \big \rangle}_{\rm HS} \, d\nu    \\
   &  \, = \,  \frac{e^{-2t}}{\sqrt { 1 - e^{-2t}}}
   \int_{\R^d} \int_{\R^d} \Big [ \big (\tau_\nu (x) - {\rm Id} \big ) y
         \cdot \nabla  v_t   \big (e^{-t } x + \sqrt {1 - e^{-2t}} \, y \big ) \Big ]
            d\nu(x) d \gamma (y) \\
\end {split} \end {equation*}
which is \eqref{e:nps}.
To deduce the estimate \eqref {eq1.9}, it suffices to apply (twice) the
Cauchy-Schwarz inequality to the right-hand side of \eqref{e:nps}
in such a way that, by integrating out the $y$ variable,
\begin{equation*} \begin {split}
 {\rm I}_\gamma(P_t h)
   &  \, \leq \,    \frac{e^{-2t}}{\sqrt { 1 - e^{-2t}}}
      \int_{\R^d} \! \int_{\R^d}  \big | \big (\tau_\nu (x) - {\rm Id} \big ) y \big |
         \big | \nabla  v_t   \big (e^{-t } x + \sqrt {1 - e^{-2t}} \, y \big ) \big |
            d\nu (x)  d\gamma (y) \\
  & \, \leq \,   \frac{e^{-2t}}{\sqrt { 1 - e^{-2t}}}
      \bigg (\int_{\R^d} {\| \tau_\nu - {\rm Id} \| }_{\rm HS} ^2 \, d\nu \bigg )^{1/2}
      \bigg ( \int_{\R^d} P_t \big ( | \nabla v_t |^2 \big ) d\nu \bigg)^{1/2} .
\end {split} \end {equation*}
 Since
$$
 \int_{\R^d} P_t \big ( | \nabla v_t |^2 \big ) d \nu
  \, = \, \int_{\R^d} P_t \big ( | \nabla v_t |^2 \big ) h d\gamma
     \, = \,   \int_{\R^d}  | \nabla v_t |^2 P_t h d\gamma \, = \, {\rm I}_\gamma(P_t h)
$$
 by symmetry of $P_t$, the proof of \eqref{eq1.9} is complete.

Let us now turn to the proof of \eqref{e:decaystein}.
For any smooth test function $ \varphi $ on $\R^d$, by symmetry of ${(P_t)}_{t\geq 0}$,
for any $t \geq 0$,
$$
\int_{\R^d} x \cdot \nabla \varphi  \, d\nu^t \, = \,  \int_{\R^d} x \cdot \nabla \varphi \,\, P_t h \, d\gamma
\, =\,  \int_{\R^d} P_t(x \cdot\nabla \varphi ) h d \gamma \, =\,  \int_{\R^d} P_t(x\cdot\nabla \varphi ) d\nu.
$$
By the integral representation of $P_t$,
\begin {equation*} \begin {split}
\int_{\R^d} P_t(x\cdot\nabla \varphi ) d\nu
    & \, = \,  e^{-t} \int_{\R^d}\! \int_{\R^d}
     x \cdot \nabla \varphi  \big ( e^{-t}x + \sqrt {1 - e^{-2t}} \, y \big ) d\nu (x) d \gamma (y )  \\
 & \quad + \sqrt {1 - e^{-2t}} \int_{\R^d} \! \int_{\R^d}
  y\cdot\nabla \varphi  \big ( e^{-t}x + \sqrt {1 - e^{-2t}} \, y \big ) d\nu (x) d \gamma (y ) . \\
\end {split} \end {equation*}
Use now the definition of $\tau_\nu$ in the $x$ variable and integration by parts in the
$y$ variable to get that
\begin {equation*} \begin {split}
  \int_{\R^d} P_t(x \cdot\nabla \varphi ) d\nu & \, = \,  e^{-2t} \int_{\R^d} \! \int_{\R^d}
    {\big \langle \tau_\nu (x) , \big ({\rm Hess}( \varphi) \big)  \big ( e^{-t} x+ \sqrt {1 - e^{-2t}} \, y \big )
     \big \rangle}_{\rm HS}\,\,d\nu (x) d \gamma (y )  \\
 & \quad + (1 - e^{-2t}) \int_{\R^d} \! \int_{\R^d}
   \Delta \varphi  \big ( e^{-t} x+ \sqrt {1 - e^{-2t}} \, y \big ) d\nu (x) d \gamma (y )  \\
  & \, = \,  e^{-2t} \int_{\R^d}  { \big \langle \tau_\nu ,
           P_t \big ({\rm Hess}( \varphi ) \big ) \big \rangle}_{\rm HS}\, d \nu
               + (1 - e^{-2t}) \int_{\R^d} P_t (\Delta \varphi )  d\nu  \\
    & \, = \,  e^{-2t} \int_{\R^d}  { \big \langle P_t ( h \tau_\nu ), {\rm Hess} (\varphi) \big \rangle}_{\rm HS}
       \, d \gamma  + (1 - e^{-2t})
    \int_{\R^d}  \Delta \varphi   \, P_t h \, d\gamma.  \\
\end {split} \end {equation*}
As a consequence, a Stein kernel for $\nu^t$ is
\begin{equation}\label{e:steint}
\tau_{\nu^t} \, = \,  e^{-2t} \, \frac {P_t ( h \tau_\nu )}{P_t h} +  (1 - e^{-2t}) \, {\rm Id}.
\end{equation}
Therefore,
$$
    \int_{\R^d}  {\| \tau_{\nu^t} - {\rm Id} \|}^2_{\rm HS} \, d\nu^t
  \, = \,  e^{-4t}  \int_{\R^d} \frac {{\| P_t ( h (\tau_\nu - {\rm Id}))\|}_{\rm HS}^2}{P_t h}  \, d \gamma.
$$
By the Cauchy-Schwarz inequality along $P_t$,
$$
{\big \| P_t \big( h (\tau_\nu - {\rm Id}) \big ) \big \|}_{\rm HS}^2
   \, \leq \,  P_t \big ( h {\| \tau_\nu - {\rm Id}\|}^2_{\rm HS}  \big) P_t h.
$$
Hence,
\begin {equation*} \begin {split}
    \int_{\R^d}  {\| \tau_{\nu^t} - {\rm Id} \|}^2_{\rm HS} \, d\nu^t
    \, &\leq \   e^{-4t} \int P_t \big ( h {\| \tau_\nu - {\rm Id} \|}_{\rm HS} ^2  \big) d\gamma \\
    \, &= \, e^{-4t} \int_{\R^d} {\| \tau_\nu - {\rm Id} \|}_{\rm HS} ^2 \, h d\gamma
    \, = \,  e^{-4t} \int_{\R^d} {\| \tau_\nu - {\rm Id} \|}_{\rm HS} ^2 \, d\nu,
\end {split} \end {equation*}
that is the announced result {\it (iv)}. Proposition \ref{p:up} is established.

\end{proof}


\begin{remark} \label {r:conditioning}
For every $t>0$, it is easily checked that the
mapping $x\mapsto \tau_{\nu^t}(x)$ appearing in \eqref{e:steint} admits the
probabilistic representation
\begin{equation}\label{e:condtau}
\tau_{\nu^t}(x) \, = \,  \E \big [e^{-2t} \tau_\nu(F) + (1-e^{-2t})\,{\rm Id}\, | \, F_t = x \big]
 \quad d\nu^t(x){-\rm a.e.},
\end{equation}
where, on some probability space $(\Omega, {\cal F}, \P)$,
$F$ has distribution $\nu$ and $F_t = e^{-t}F+\sqrt{1-e^{-2t}} Z$,
with $Z$ a $d$-dimensional vector with distribution $\gamma$, independent of $F$.
\end{remark}

We are now in a position to prove Theorem \ref{t:hsi}.\\

\noindent{\it Proof of Theorem \ref{t:hsi}}.
As announced, on the basis of the interpolation \eqref{eq1.7}, we
apply \eqref{e:decay} and \eqref{eq1.9} respectively to bound the Fisher information
${\rm I}_\gamma (P_t h)$ for $t$ around $0$ and away from $0$. We thus get, for every $u>0$,
\begin{equation*} \begin {split}
  {\rm H} \big (\nu \, | \, \gamma \big )
     & \, = \,   \int_0^u {\rm I}_\gamma  (P_t h)   dt + \int_u^\infty {\rm I}_\gamma  (P_t h)   dt \\
     & \,  \leq \,  {\rm I} \big (\nu \, | \, \gamma \big)   \int_0^u e^{-2t}  dt
        + {\rm S}^2 \big (\nu \, | \, \gamma \big )  \int_u^\infty \frac{e^{-4t}}{1 - e^{-2t} } \, dt \\
       & \, \leq \, \frac {1}{2} \,  {\rm I} \big (\nu \, | \, \gamma \big)   (1 - e^{-2u})
        + \frac {1}{2} \, {\rm S}^2 \big (\nu \, | \, \gamma \big )
           \big ( - e^{-2u} - \log ( 1 - e^{-2u}) \big ) .
\end {split} \end {equation*}
Optimizing in $u$ (set $ 1 - e^{-2u} = r \in (0,1)$) concludes the proof. \qed

\begin{remark} \label {r:expdecay}

It is worth mentioning that a slight modification of the proof of {\it (iii)} in Proposition~\ref {p:up}
leads to the improved form of the exponential decay \eqref {e:decay} of the Fisher information
\begin{equation}\label{e:decayimproved}
 {\rm I} \big (\nu^t \, | \, \gamma \big )
  \, \leq \,  \frac {e^{-2t} \,  {\rm S}^2 (\nu \, | \, \gamma ) \, {\rm I} (\nu \, | \, \gamma )}
      { {\rm S}^2 (\nu \, | \, \gamma ) + (e^{2t} - 1) \, {\rm I} (\nu \, | \, \gamma  )} \, .
\end{equation}
As for the classical logarithmic Sobolev inequality, the inequality
\eqref {e:decayimproved} may be integrated
along de Bruijin's formula \eqref {e:bruijn} towards
the better, although less tractable, HSI inequality
$$
 {\rm H}  \, \leq \, \frac {  {\rm S}^2 \,  {\rm I} } {2 (  {\rm S}^2 -  {\rm I}) } \,
    \bigg ( 1 +  \frac {  {\rm I}  } {  {\rm S}^2 -  {\rm I} } \,
        \log \Big ( \frac {  {\rm I}  } { {\rm S}^2  } \Big ) \bigg )
$$
(understood in the limit as ${\rm S}^2 = {\rm I}$), where
$ {\rm H} = {\rm H} (\nu \, | \, \gamma )$, $ {\rm S} = {\rm S} (\nu \, | \, \gamma )$
and $ {\rm I} = {\rm I} (\nu \, | \, \gamma )$.

\end {remark}

Together with the de Bruijn identity \eqref {e:bruijn}, the classical logarithmic
Sobolev inequality \eqref {e:logsob} ensures the exponential decay in $ t \geq 0$ of the relative entropy
\begin {equation} \label {e:decay1}
 {\rm H} \big (\nu^t \, | \, \gamma \big ) \, \leq \, e^{-2t} \,  {\rm H} \big (\nu^0 \, | \, \gamma \big )
\end {equation}
along the Ornstein-Uhlenbeck semigroup (cf.~e.g.~\cite[Theorem~5.2.1]{B-G-L}). The new HSI produces
a reinforcement of this exponential convergence to equilibrium
under finiteness of the Stein discrepancy.

\begin {corollary} [Exponential decay of entropy from HSI]
\label {cor.decay}
Let $\nu $ with Stein discrepancy $  {\rm S} (\nu \, | \, \gamma  ) = {\rm S}$.
For any $ t \geq 0$,
\begin {equation} \label {e:decay2bis}
 {\rm H} \big (\nu^t \, | \, \gamma \big ) \, \leq \, 
    \frac {e^{-4t}} {e^{-2t} + \frac {1 - e^{-2t}}{{\rm S}^2} \, {\rm H} (\nu^0 \, | \, \gamma ) } 
     \, {\rm H} \big (\nu^0 \, | \, \gamma \big)
     \, \leq \, \frac {e^{-4t}}{1 - e^{-2t}} \, {\rm S}^2 \big (\nu^0 \, | \, \gamma \big ) .
\end {equation}
\end {corollary}

\begin {proof}
Together with \eqref {e:decaystein} and since $r\mapsto r\log \big (1+ \frac {s}{r} \big )$ is increasing for any fixed $s$, the HSI inequality applied to $\nu^t$ implies that
$$
 {\rm H} \big (\nu^t \, | \, \gamma \big ) \, \leq \,
  \frac {e^{-4t} \, {\rm S}^2}{2} \, 
      \log \bigg ( 1 + \frac {e^{4t} \, {\rm I} (\nu^t \, | \, \gamma )}{{\rm S}^2} \bigg ) .
$$
Set $ U(t) = \frac{e^{4t}}{{\rm S}^2} \, {\rm H} (\nu^t \, | \, \gamma )$, $t \geq 0$, so that by
\eqref {e:bruijn}, $ U' = 4 U - \frac {e^{4t}}{{\rm S}^2} \, {\rm I} (\nu^t \, | \, \gamma )$. The latter
inequality therefore rewrites as
\begin {equation} \label {e:differential}
e^{2U} - 1 - 4U \, \leq \, - U'.
\end {equation}
Since $ e^r - 1 - r \geq \frac {r^2}{2}$ for $r \geq 0$, this inequality may be relaxed
into $ { - 2U + 2U^2 \, \leq \, - U'}$. Setting $V(t) = e^{-2t} U(t)$, $ t \geq 0$, it follows that
$2 e^{2t} V^2(t) \leq - V'(t)$ so that, after integration,
$$
e^{2t} - 1 \, \leq \, \frac {1}{V(t)} - \frac {1}{V(0)} \, .
$$
By definition of $V$, this inequality amounts to the conclusion of Corollary~\ref {cor.decay} and the proof is complete.
\end {proof}

\subsection {Stein discrepancy and concentration inequalities}

This paragraph investigates another feature of Stein's discrepancy applied to
concentration inequalities. It is of course by now classical that logarithmic Sobolev
inequalities may be used as a robust tool towards (Gaussian) concentration inequalities
(cf.~e.g.~\cite {L2,B-L-M}). For example, for the standard Gaussian measure $\gamma$
itself, the Herbst argument yields that for any $1$-Lipschitz function $ u : \R^d \to \R$
with mean zero,
\begin {equation} \label {eq.gaussianconcentration}
\gamma ( u \geq r ) \, \leq \, e^{-r^2/2}, \quad r \geq 0.
\end {equation}
Equivalently (up to numerical constants) in terms of moment growth,
\begin {equation} \label {eq.gaussianmoment}
\bigg (\int_{\R^d} |u|^p d \gamma \bigg)^{1/p} \, \leq \, C {\sqrt p} \, , \quad p \geq 1.
\end {equation}

Here, we describe how to directly implement Stein's discrepancy into such concentration
inequalities on the basis of the principle leading to the HSI inequality.
If $\nu $ is a probability measure on the Borel sets of $\R^d$ with
Stein kernel $\tau_\nu$, set for $p \geq 1$,
$$
{\rm S}_p \big ( \nu \, | \, \gamma \big ) \, = \,
      \bigg (\int_{\R^d} {\| \tau _\nu - {\rm Id}\|}_{\rm HS}^p \, d\nu \bigg)^{1/p}.
$$
Hence $ {\rm S}_2 ( \nu \, | \, \gamma  ) = {\rm S}  ( \nu \, | \, \gamma  )$
is the Stein discrepancy as defined earlier.
Recall ${\|\cdot \|}_{\rm op}$ the operator norm on the $d \times d $ matrices.

\begin {theorem} [Moment bounds and Stein discrepancy]
\label {thm.concentration}
Let $\nu $ have Stein kernel $\tau_\nu$.
There exists a numerical constant $ C>0$ such that for every $1$-Lipschitz
function $ {u : \R^d \to \R}$ with $\int_{\R^d} u d\nu = 0$, and every $ p \geq 2$,
\begin {equation} \label {eq.moment}
\bigg (\int_{\R^d} |u|^p d \nu \bigg)^{1/p}
     \, \leq \,  C  \bigg [  {\rm S}_p \big ( \nu \, | \, \gamma \big )
      + {\sqrt p} \, \bigg (\int_{\R^d} {\| \tau_\nu \|}_{\rm op}^{p/2} \, d\nu \bigg)^{1/p} \bigg ].
\end {equation}
\end {theorem}

Before turning to the proof of this result, let us comment on its measure concentration content.
One first important aspect is that the constant $C$ is dimension free.
When $\nu = \gamma$, \eqref {eq.moment} exactly fits the Gaussian case
\eqref {eq.gaussianmoment}. In general, the moment growth in $p$ describes various
concentration regimes of $\nu$  (cf.~\cite[Section~1.3]{L2}, \cite[Chapter~14]{B-L-M})
according to the growth of the $p$-Stein discrepancy ${\rm S}_p  ( \nu \, | \, \gamma )$.

In view of the elementary estimate
${\|\tau_\nu\|}_{\rm op} \leq 1+ {\| \tau_\nu - {\rm Id}\|}_{\rm HS}$,
the conclusion \eqref {eq.moment} immediately yields the moment growth
\eqref {eq.moments2} emphasized in the introduction
$$
\bigg (\int_{\R^d} |u|^p d \nu \bigg)^{1/p} \, \leq \,  C  \Big (  {\rm S}_p \big ( \nu \, | \, \gamma \big ) 
   + {\sqrt p} + {\sqrt p} \, {\sqrt { {\rm S}_p \big ( \nu \, | \, \gamma \big ) }} \, \Big).
$$
Note that there is already an interest to write this bound for $p=2$, 
$$
{\rm Var}_\nu (u ) \, \leq \, C   \Big ( 1 + {\rm S} \big ( \nu \, | \, \gamma \big )
   +  {\rm S} ^2 \big ( \nu \, | \, \gamma \big )  \,\Big).
$$
Together with E. Milman's Lipschitz characterization of Poincar\'e inequalities for
log-concave measures \cite {M}, it shows that the Stein discrepancy
$  {\rm S} (\nu \, | \, \gamma  ) $ with respect to the standard Gaussian measure is another
control of the spectral properties in this class of measures.

Similar inequalities hold for arbitrary covariances by suitably adapting the Stein kernel as in
Corollary~\ref {c:hsi}.

A main example of illustration of Theorem~\ref {thm.concentration}
concerns sums of independent random vectors.
Consider $X$ a mean zero random variable on a probability space $(\Omega, {\cal F}, \P)$
with values in $\R^d$, and $X_1, \ldots , X_n$ independent copies
of $X$. Assume that the law $\nu $ of $X$ admits a Stein kernel
$\tau_{\nu}$.
Setting $ T_n = \frac {1}{\sqrt n} \sum_{k=1}^n X_k $, it is easily
seen by independence that, as matrices, a Stein kernel $\tau_{\nu_n}$
of the law $\nu_n$ of $T_n$ satisfies
$$
\tau_{\nu_n} (T_n) 
   \, = \,   \E  \bigg ( \frac {1}{n} \sum_{k=1}^n \tau_\nu (X_k) \, \Big | \, T_n \bigg).
$$
Hence, 
$$ 
{\rm S}_p \big ( \nu_n \, | \, \gamma \big ) \, \leq \, 
\E \bigg ( \bigg \| \frac {1}{n} \, 
   \sum_{k =1}^n \big [\tau_\nu (X_k) - {\rm Id} \big ] 
   \bigg \|_{\rm HS}^p \bigg )^{1/p}.
$$
By the triangle inequality, the latter is bounded from above by
$$
\E \big ( {\| \tau_\nu (X) - {\rm Id} \|}_{\rm HS}^p\big)^{1/p}
   \, =\, {\rm S}_p \big ( \nu \, | \, \gamma \big) \, = \, {\rm S}_p
$$
which produces a first bound of interest.
If it is assumed in addition that the covariance matrix of $X$ is the identity, we may use
classical inequalities for sums of independent centered random
vectors (in Euclidean space) to the family $ \tau_\nu (X_k) - {\rm Id}$,
$k = 1, \ldots, n$. Hence, by for example Rosenthal's inequality
(see~e.g.~\cite{B-L-M, MJCFT}), for $ p \geq 2$,
$$ 
\E \bigg ( \bigg \| \frac {1}{n} \, 
   \sum_{k =1}^n \big [\tau_\nu (X_k) - {\rm Id} \big ] 
   \bigg \|_{\rm HS}^p \bigg )^{1/p}
   \, \leq \, K_p  \, n^{-1/2} \, {\rm S}_p .
$$
Together with \eqref {eq.moments2}, it yields a growth control of the moments of 
$ u (T_n)$ for any Lipschitz function $u$, and therefore concentration of the law of $T_n$.
More precisely, and since it is known that
$K_p = O(p)$, for any $1$-Lipschitz function $u : \R^d \to \R$ such that $\E(u(T_n)) = 0$,
$$
\E \big ( \big | u(T_n) \big |^p \big)^{1/p} \, \leq \, C \, {\sqrt p} \, 
      \Big ( 1 + n^{-1/2} \, {\sqrt p} \; { \rm S}_p  + n^{-1/4}  {\sqrt {p \, {\rm S}_p}} \, \Big)
$$
for some numerical $C>0$. Note that the bound is optimal both for $\nu = \gamma$
and as $n \to \infty$ describing the standard Gaussian
concentration \eqref {eq.gaussianmoment}.
By Markov's inequality, optimizing in $p \geq 2$, one deduces that for some numerical $C' >0$,
\begin {equation} \label {e:concentrationrn}
\P \big ( u (T_n) \geq r \big ) \, \leq \, C' \, e^{-r^2 /C'}
\end {equation}
for all $0 \leq r \leq r_n$ where $r_n \to \infty$
according to the growth of ${\rm S}_p$ as $p\to \infty$. For example, if
${\rm S}_p = O(p^\alpha)$ for some $\alpha >0$ (see below for such illustrations), then 
$$
\E \big ( \big | u(T_n) \big |^p \big)^{1/p} \, \leq \, C \, {\sqrt p} 
$$
(for some possibly different numerical $C>0$) for every $p \leq n^{\frac {1}{2\alpha + 2}}$.
By Markov's inequality in this range of $p$,
$$
\P \big ( \big | u(T_n) \big| \geq r \big ) \, \leq \, \Big ( \frac {C\, {\sqrt p}}{r} \Big )^p,
$$
and with $p \sim \frac {r^2}{4C^2}$, the claims follows with $r_n$
of the order of $n^{\frac {1}{4\alpha + 4}}$.

For the applications of the concentration inequality
\eqref {e:concentrationrn}, it is therefore useful to provide a handy set of conditions ensuring
a suitable control of (the growth in $p$ of) ${\rm S}_p = S _p ( \nu \, | \, \gamma)$,
that is of the moments of the Stein kernel
$\tau_\nu (X)$ of a given random variable
$X$ with law $\nu$. The following remark collects families of examples in dimension one.
Together with this remark, \eqref {e:concentrationrn} therefore produces
with the Stein methodology concentration properties for measures not necessarily
satisfying a logarithmic Sobolev inequality.
For example, the conclusion may be applied to a vector $X$ 
with independent coordinates in $\R^d$ each of them of the Pearson class
as described in (b) of the following Remark~\ref {r:pearson}.

\begin{remark} \label {r:pearson}
For concreteness, we describe two classes of one-dimensional
distributions such that the associated Stein kernel has finite moments of all orders.
Denote by $X$ a centered real-valued random variable with law
$\nu $ and Stein kernel
$\tau_\nu$. Recall from \eqref {e:kerneldimensionone} of Remark~\ref {r:kernel}, that if
$\nu $ has density $\rho $ with respect to the Lebesgue measure, 
a version of  $\tau_\nu $ is given by
$\tau_\nu (x) = \rho(x)^{-1}\int_x^\infty y \rho (y)dy$
for $x$ inside the support of $\rho$.

\begin{enumerate}

\item [(a)] Assume that $\rho (x)=q(x)\frac{e^{-x^2/2}}{\sqrt{2\pi}}$,
$x\in \R$, where $q$ is  smooth and satisfies the uniform bounds
$q(x)\geq c>0$ et $|q'(x)|\leq C<\infty$ for
constants $ c, C >0$. Therefore,
$$
\tau_\nu (x) \, = \, 1+\frac{e^{-x^2/2}}{q(x)}\int_x^\infty q'(y)e^{-y^2/2}dy
  \, = \, 1-\frac{e^{-x^2/2}}{q(x)}\int_{-\infty}^x q'(y)e^{-y^2/2}dy.
$$
Studying separately the two cases $x>0$ and $x<0$, it easily follows that
$ {|\tau_\nu(x)-1|} \leq \frac{\sqrt{2\pi}C}{c}$, and consequently
$\E (|\tau_\nu(X)|^r)<\infty$ for every $r>0$.

\item [(b)]  Assume that the support of $\rho$
coincides with an open interval of the type $(a,b)$,
with $-\infty\leq a<b\leq +\infty$. Say then that the law $\nu $ of $X$ is a (centered)
member of the \textit{Pearson family} of continuous
distributions if the density $\rho $ satisfies the differential equation
\begin{equation}\label{Pearson}
\frac{\rho '(x)}{\rho (x)} \, = \, \frac{a_0+a_1 x}{b_0+b_1 x+b_2 x^2} \, , \quad x\in(a,b),
\end{equation}
for some real numbers $a_0,a_1,b_0,b_1,b_2$. We refer the reader
e.g. to \cite[Sec. 5.1]{DZ} for an introduction to the
Pearson family. It is a well-known fact that there are
basically five families of distributions satisfying
(\ref{Pearson}): the centered normal distributions, centered gamma
and beta distributions, and distributions that are obtained by
centering densities of the type $\rho (x)=Cx^{-\alpha}e^{-\beta/x}$
or $\rho (x)=$ $ {C(1+x)^{-\alpha}\exp(\beta\arctan(x))}$ ($C$ being a
suitable normalizing constant). According to
\cite[Theorem 1, p. 65]{Ste}, if $\tau_\nu $ satisfies
\begin{equation}\label{tauexplosion}
\int_0^b \frac {y}{\tau_\nu (y)} \, dy \, = \,
+\infty \quad \text{and} \quad \int_a^0
 \frac {y}{\tau_\nu (y)} \, dy \, = \, -\infty,
\end{equation} 
then $\tau_\nu  (x)=\alpha x^2 +\beta x + \gamma$, $x\in(a,b)$ (with
$\alpha,\,\beta, \,\gamma$ real constants) if and only if
$\nu $ is a member of the Pearson family in the sense that $\rho $ satisfies
\eqref{Pearson} for every $x\in(a,b)$ with $a_0=\beta
$, $a_1=2\alpha + 1$, $b_0=\gamma$, $b_1=\beta$ and $b_2=\alpha$.
It follows that if $\nu $ is
centered member of the Pearson family such that \eqref{tauexplosion} is satisfied and
$X$ has finite moments of all orders,
so has $\tau_\nu (X)$. This includes the case of Gaussian, gamma and beta
distribution for example.
\end{enumerate}
\end{remark}

Further illustrations of Theorem~\ref {thm.concentration} 
may be developed in the general context of eigenfunctions on
abstract Markov Triples $(E, \mu, \Gamma)$ as addressed in the forthcoming Section 5.
Indeed, let $F : E \to \R $ be an eigenfunction of the underlying diffusion
operator ${\rm L} $ with eigenvalues $\lambda >0$ with distribution $\nu$
and normalized such that $\int_E F^2 d\mu = 1$.
Then, according to Proposition~\ref {p:smg} below, a version of the Stein kernel is given by
$$
\tau_\nu \, = \,   \frac {1}{\lambda} \, \E_\mu \big ( \Gamma (F)\, | \, F \big )
$$
so that
$$
{\rm S}_p ^p \big ( \nu \, | \, \gamma \big ) 
      \, \leq \,   \int_E  \bigg | \frac {\Gamma (F)}{\lambda}  - 1 \bigg | ^p d\mu.
$$
In concrete instances, such as Wiener chaos for example, the latter expression may be
easily controled so to yield concentration properties of the underlying distribution of $F$.
For example, in the setting of the recent \cite{A-C-P}, it may be shown
by hypercontractive means that for
the Hermite, Laguerre or Jacobi (or mixed ones) chaos structures, for any $p \geq 2$,
$$
{\rm S}_p ^p \big ( \nu \, | \, \gamma \big ) 
      \, \leq \,C_{p,\lambda}  \left( \int_E F^4d\mu - 3\right)^{p/2}.
$$
According to the respective growth in $p$ of $C_{p,\lambda} $, concentration properties
on $F$ may be achieved.

We turn to the proof of Theorem~\ref {thm.concentration}.

\begin {proof}
We only prove the result for $p$ an even integer, the general case following similarly with
some further technicalities. We may also replace the assumption
$\int_{\R^d} u d\nu = 0$ by $\int_{\R^d} u d\gamma = 0$ by a simple use of the triangle inequality.
Indeed, by Jensen's inequality,
$$
\bigg | \int_{\R^d} u \, d\nu  - \int_{\R^d} u \, d\gamma \bigg |^p
    \, \leq \, \int_{\R^d} \bigg |  u - \int_{\R^d} u \, d\gamma \bigg | ^p d \nu
$$
so that if the conclusion~\eqref {eq.moment} holds for $u$ satisfying
$\int_{\R^d} u d\gamma = 0$, it holds similarly for $u$ satisfying
$\int_{\R^d} u d\nu = 0$ with maybe $2C$ instead of $C$.

We run as in the preceding section the Ornstein-Uhlenbeck semigroup ${(P_t)}_{t \geq 0}$
with infinitesimal generator ${\cal L} = \Delta - x \cdot \nabla $.
Let $u : \R^d \to \R$ be $1$-Lipschitz, assumed furthermore to be smooth
and bounded after a cut-off
argument (cf.~\cite [Section~1.3]{L2} for standard technology in this regard).
Let thus $q \geq 1$ be an integer, and set
$$
\phi (t) \,=\, \int_{\R^d} (P_t u)^{2q} d \nu, \quad t \geq 0. 
$$
Under the centering hypothesis $\int_{\R^d} u d\gamma = 0$, $\phi (\infty) = 0$.
Differentiating along ${(P_t)}_{t \geq 0}$ together with the definition of Stein
kernel $\tau_\nu$ yields
\begin{equation*} \begin {split}
\phi'(t) 
     & \,=\,  2q \int_{\R^d} (P_t u)^{2q-1} \, {\cal L} P_t u \, d \nu \\
    & \,=\, 2q \int_{\R^d} (P_t u)^{2q-1} \, \Delta P_t u \, d \nu 
           - \int_{\R^d} x \cdot \nabla (P_t u)^{2q} d\nu \\
       & \,=\,  2q \int_{\R^d} (P_t u)^{2q-1} \, \Delta P_t u \,  d \nu 
           - \int_{\R^d}  \big \langle \tau_\nu ,  {\rm Hess} \big (  (P_t u)^{2q} \big )
                 \big \rangle _{\rm HS} \, d \nu \\
        & \,=\,  2q \int_{\R^d} (P_t u)^{2q-1} 
                 \big \langle {\rm Id} - \tau_\nu , {\rm Hess} (P_t u) \rangle_{\rm HS} \,d \nu  \\
          & \quad - 2q (2q-1)\int_{\R^d} (P_t u)^{2q -2} 
                  \langle \tau_\nu ,  \nabla P_t u \otimes \nabla P_t u   \rangle _{\rm HS} \, d \nu . \\
\end {split} \end {equation*}
As in the proof of Proposition~\ref {p:up},
\begin{equation*} \begin {split}
   & \int_{\R^d} {\big  (P_t u)^{2q-1}
   \langle  \tau_\nu  -  {\rm Id} ,{\rm Hess}(P_t u) \big \rangle}_{\rm HS} \, d\nu    \\
   &  \, = \,  \frac{e^{-2t}}{\sqrt { 1 - e^{-2t}}}
   \int_{\R^d} \int_{\R^d} (P_t u)^{2q-1}(x)\,\big (\tau_\nu (x) - {\rm Id} \big ) y
         \cdot \nabla  u   \big (e^{-t } x + \sqrt {1 - e^{-2t}} \, y \big ) 
            d\nu(x) d \gamma (y) \\
\end {split} \end {equation*}
Using that $|\nabla u | \leq 1$ (since $u$ is $1$-Lipschitz) and furthermore
$$
| \nabla P_t u | \, \leq \, e^{-t} P_t \big ( |\nabla u | \big ) \, \leq \, e^{-t} ,
$$
it easily follows as in the previous section that for every $ t $,
\begin{equation} \begin {split} \label {eq.phi}
  - \phi'(t) 
  & \, \leq \, \frac {e^{-2t}}{\sqrt {1 - e^{-2t}}} 
       \int_{\R^d} 2q |P_t u|^{2q-1} {\|\tau_\nu - {\rm Id} \|}_{\rm HS} \, d\nu \\
    & \quad + e^{-2t} \int_{\R^d} 2q (2q-1) (P_t u)^{2q-2} {\|\tau_\nu  \|}_{\rm op} \, d\nu . \\
\end {split} \end {equation}
By the Young-H\"older inequality,
$$
2q |P_t u|^{2q-1} {\|\tau_\nu - {\rm Id} \|}_{\rm HS}
   \, \leq \,  \frac {1}{\alpha} \, {\| \tau_\nu - {\rm Id} \|}_{\rm HS} ^\alpha 
   + \frac {1}{\beta} \, \big [ 2q |P_t u|^{(2q-1)} \big]^ \beta
$$
where $\alpha = 2q$ and $(2q-1) \beta = 2q$, and
$$
2q (2q-1) (P_t u)^{2q-2} {\|\tau_\nu  \|}_{\rm op}
   \, \leq \,  \frac {1}{\alpha'} \, \big [ (2q-1) {\|\tau_\nu  \|}_{\rm op} \big ]^{\alpha'} 
   + \frac {1}{\beta'} \,\big [ 2q (P_t u)^{(2q-2)} \big] ^{\beta '}
$$
where $\alpha' = q$ and $(2q-2) \beta' = 2q$. Therefore \eqref {eq.phi} implies that,
for every $t$, 
$$
- \phi'(t) \, \leq \,   C(t) \, \phi \, (t) + D(t)
$$
where 
$$
C(t) \, = \,  \frac {e^{-2t}}{\sqrt {1 - e^{-2t}}} \, (2q)^\beta + e^{-2t} \, (2q)^{\beta '}
$$
and
$$
D(t) \, = \,  \frac {e^{-2t}}{\sqrt {1 - e^{-2t}}} 
    \int_{\R^d}  {\|\tau_\nu - {\rm Id} \|}_{\rm HS} ^{2q} \, d\nu
       + e^{-2t}   \int_{\R^d} \big [ (2q-1) {\|\tau_\nu  \|}_{\rm op} \big ]^q  d\nu.
$$
Integrating this differential inequality yields that
$$
\phi (t) \, \leq \,  e^{{\widetilde C}(t)} \int_t ^\infty  e^{-{\widetilde C}(s)} D(s) ds 
$$
where $ {\widetilde C}(t) = \int_t ^\infty C(s) ds$, $ t \geq 0$. It follows that 
$ \phi (0) \leq  e^{{\widetilde C}(0)} \int_0 ^\infty  D(s) ds $
and therefore
$$
\int_{\R^d} |u|^{2q} d\nu  \, = \, \phi (0)
   \, \leq \,  e^{{\widetilde C}(0)} 
   \bigg ( \int_{\R^d}  {\|\tau_\nu - {\rm Id} \|}_{\rm HS} ^{2q} \,  d\nu
       +  \int_{\R^d} \big [ (2q-1) {\| \tau_\nu  \|}_{\rm op} \big ]^q  d\nu \bigg ).
$$
Since ${\widetilde C}(0)$ is bounded above by $Cq$ for some numerical $C>0$,
the announced claim follows.
The proof of Theorem~\ref {thm.concentration} is therefore complete.
\end {proof}

\subsection {On the rate of convergence in the entropic central limit theorem}

In this last paragraph, we provide a brief and simple application of the HSI inequality
to (yet non optimal) rates in the entropic central limit theorem.
Let $X$ be a real-valued random variable on a probability space $(\Omega, {\cal F}, \P)$
with mean zero and variance one. Let also $X_1, \ldots , X_n$ be independent copies
of $X$ and set
$$
T \, = \, \sum_{k=1}^n a_k X_k
$$
where $\sum_{k=1}^n a_k^2 = 1$.

Assume that the law $\nu $ of $X$ has a density $h$ with respect to the standard
Gaussian measure $\gamma $ on $\R$ with (finite) Fisher information
${\rm I} (\nu \, | \, \gamma)$
and a Stein kernel $\tau_\nu $ with discrepancy
${\rm S} (\nu \, | \, \gamma)$. Let $\nu_T$ be the law of $T$.
The classical Blachman-Stam inequality (cf.~\cite{Sta,Bl,V}) indicates that
$$
{\rm I} \big (\nu_T \, | \, \gamma \big ) \, \leq \, {\rm I} \big (\nu \, | \, \gamma \big ).
$$
On the other hand, as in the previous paragraph,
$$
\tau_{\nu_T} (T) \, = \,  \E  \bigg ( \sum_{k=1}^n a_k^2 \tau_\nu (X_k) \, \Big | \, T \bigg)
$$ 
so that
$$
{\rm S}^2 \big (\nu_T \, | \, \gamma \big ) \, \leq \,  \alpha (a) \, {\rm S}^2 \big (\nu \, | \, \gamma \big )
$$
where $ \alpha (a) = \sum_{i=1}^n a_i^4$.

As a consequence therefore of the HSI inequality of Theorem~\ref {t:hsi},
\begin {equation} \label {eq.rate}
{\rm H} \big (\nu_T \, | \, \gamma \big ) \, \leq \,
 \frac {1}{2} \, \alpha (a) \, {\rm S}^2 \big (\nu \, | \, \gamma \big )
     \log \bigg ( 1 + \frac {{\rm I} (\nu \, | \, \gamma)}{\alpha (a) \, {\rm S}^2 (\nu \, | \, \gamma)} \bigg).
\end {equation}
This result has to be compared with the works \cite {A-B-B-N} and 
\cite {B-J} (cf.~\cite {J}) which produce the bound
\begin {equation} \label {eq.ratepoincare}
{\rm H} \big(\nu_T \, | \, \gamma\big ) \, \leq \,  \frac {\alpha (a)}{c/2 + (1-c/2) \alpha (a)} 
\, {\rm H} \big (\nu \, | \, \gamma \big )
\end {equation}
under the hypothesis that $\nu$ satisfies a Poincar\'e inequality with constant $c>0$.

For the classical average $ T_n  =  \frac {1}{\sqrt n} \sum_{k=1}^n  X_k$,
\eqref {eq.ratepoincare} yields a rate $O(\frac {1}{n})$ in the entropic
central limit theorem while 
\eqref {eq.rate} only produces $O(\frac {\log n}{n})$, however at a
cheap expense and under potentially different conditions as described in
Remark~\ref {r:pearson}.
For this classical average, the recent works \cite {B-C-G1,B-C-G2} actually provide a complete
picture with rate $O(\frac {1}{n})$
under a fourth-moment condition on $X$ based on local central limit theorems and Edgeworth expansions.
General sums $T = \sum_{k=1}^n a_k X_k$ are studied in \cite {B-C-G3} as a
particular case of sums of independent non-identically distributed random variables.
Vector-valued random variables may be considered similarly.

\section {Transport distances and Stein discrepancy}\label{S3}

In this section, we develop further inequalities involving the Stein discrepancy,
this time in relation with Wasserstein distances. A new improved form of
the Talagrand quadratic transportation cost inequality, called WSH, is emphasized,
and comparison between the HSI inequality and the Talagrand
and Otto-Villani HWI inequalities is provided.
Let again $\gamma = \gamma^d $ denote
the standard Gaussian measure on $\R^d$.

Fix $p\geq 1$. Given two probability measures
$\nu$ and $\mu$ on the Borel sets of $\R^d$ whose marginals have finite absolute moments
of order $p$, define the {\it Wasserstein distance} (of order $p$)
between $\nu$ and $\mu$ as the quantity
$$
{\rm W}_p(\nu,\mu) \, = \, \inf_\pi 
\bigg ( \int_{\R^d\times\R^d} |x-y|^p d\pi(x,y)\bigg ) ^{1/p}
$$
where the infimum runs over all probability measures $\pi$ on $\R^d\times \R^d$ with marginals
$\nu$ and $\mu$.
Relevant information about Wasserstein
(or Kantorovich) distances can be found, e.g. in \cite[Section~I.6]{V}.

We shall subdivide the analysis into two parts. In Section~\ref{ss:W2},
we deal with the special case of the quadratic Wasserstein distance ${\rm W}_2$, for which
we use the definition \eqref {e:steinmatrix} of a Stein kernel.
In Section~\ref{ss:W3}, we deal with general Wasserstein distances ${\rm W}_p$
possibly of order $p\neq 2$, for which it seems necessary to use the stronger definition
\eqref{e:strongersm} adopted in \cite{N-P-S, N-P-S2}.

\subsection{The case of the Wasserstein distance ${\rm W}_2$}\label{ss:W2}

We provide here a dimension-free estimate on the Wasserstein ${\rm W}_2$ distance
expressed in terms of the Stein discrepancy. In the forthcoming statement,
denote by $\nu$ a centered probability measure on $\R^d$ admitting a Stein kernel
$\tau_\nu$ (that is, $\tau_\nu$ verifies
\eqref{e:steinmatrix} for every smooth test function $\varphi$). It is not
assumed that $\nu$ admits a density with respect to the Lebesgue
measure on $\R^d$ (in particular, $\nu$ can have atoms). As already observed,
the existence of a Stein kernel for $\nu$ implies that $\nu$ has finite moments of order 2.

\begin {proposition} [Wasserstein distance and Stein discrepancy] \label{p:ws}
For every centered probability measure $\nu $ on $\R^d$,
\begin {equation} \label {e:w2s}
{\rm W}_2(\nu, \gamma) \, \leq \,  {\rm S} \big ( \nu \, | \, \gamma  \big ).
\end {equation}
\end {proposition}

\begin{proof}
Assume first that $d\nu = h d\gamma$ where $h$ is a smooth density
with respect to the standard Gaussian measure $\gamma$ on $\R^d$. As
in Section~\ref {S2}, write
$v_t = \log P_t h$ and ${d\nu^t = P_th d\gamma}$.
We shall rely on the estimate, borrowed from \cite[Lemma~2]{O-V} (cf.~also
\cite[Theorem~24.2(iv)]{V}),
\begin{equation}\label{ottovillani}
 \frac{d^+}{dt} \,  {\rm W}_2(\nu,\nu^t)
 \, \leq \,  \left( \int_{\R^d} \left|\nabla v_t  \right|^2 d\nu^t \right)^{1/2}.
\end{equation}
Note that \eqref {ottovillani} is actually the central argument
in the Otto-Villani theorem \cite {O-V} asserting that a logarithmic Sobolev inequality implies
a Talagrand transport inequality.
Here, by making use of \eqref{ottovillani} and then \eqref{eq1.9} we get that
$$
{\rm W}_2(\nu,\gamma)
 \, \leq \, \int_0^\infty \bigg (\int_{\R^d} | \nabla v_t|^2d \nu^t\bigg )^{1/2}dt
 \, \leq \, {\rm S} \big (\nu\,|\,\gamma \big )\,\int_0^\infty \! \frac{e^{-2t}}{\sqrt{1-e^{-2t}}} \, dt
$$
which is the result in this case.

The general case is obtained by a simple regularization procedure which is best presented
in probabilistic terms. Fix $\varepsilon >0 $
and introduce the auxiliary random variable $F_\varepsilon = e^{-\varepsilon} F+ \sqrt{1-e^{-2\varepsilon}} Z$
where $F$ and $Z$ are independent with respective laws $\nu$ and $\gamma$.
It is immediately checked that: (a) the distribution of $F_\varepsilon$, denoted by 
$\nu^\varepsilon$,
admits a smooth density $h_\varepsilon$ with respect to $\gamma$ (of course,
this density coincides with $P_\varepsilon h$
whenever the distribution of $F$ admits a density $h$ with respect to $\gamma$
as in the first part of the proof); (b) a Stein kernel for $\nu^{\varepsilon}$ is given by
$$
\tau_{\nu^\varepsilon}(x) \, = \,  \E \big [e^{-2\epsilon}\tau_\nu(F)
   + (1-e^{-2\varepsilon}) \, {\rm Id}\, | \, F_\varepsilon = x \big ] \quad
   d\nu^\varepsilon(x) {-\rm a.e.}
$$
(consistent with \eqref{e:condtau});
(c) ${\rm S}(\nu^\varepsilon \, |\, \gamma) 
\leq e^{-2\varepsilon} \, {\rm S}(\nu \, |\, \gamma)$;
(d) as $\varepsilon \to 0$, $F_\varepsilon$ converges to
$F$ in ${\rm L}^2$, so that, in particular,
${\rm W}_2(\nu^\varepsilon, \gamma) \to {\rm W}_2(\nu, \gamma)$.
One therefore infers that
$$
{\rm W}_2(\nu, \gamma) \, = \,  \lim_{\varepsilon\to 0}  {\rm W}_2(\nu^\varepsilon, \gamma)
  \, \leq \,  \limsup_{\varepsilon\to 0} \,  {\rm S} \big (\nu^\varepsilon\,|\, \gamma \big )
  \, \leq \,   {\rm S} \big (\nu \, |\, \gamma \big) ,
$$
and the proof is concluded.
\end{proof}

The inequality \eqref{e:w2s} may of course be compared to the Talagrand
quadratic transportation cost inequality \cite{T,V,B-G-L}
\begin {equation} \label {e:talagrand}
{\rm W}_2 ^2(\nu , \gamma )  \, \leq \,  2 \,  {\rm H} \big ( \nu \, | \, \gamma  \big ).
\end {equation}
As announced in the introduction, one can actually further refine
\eqref{e:w2s} in order to deduce an improvement of \eqref {e:talagrand}
in the form of a WSH inequality. The refinement relies on the HSI inequality itself.

\begin{theorem}[Gaussian WSH inequality] \label {t:wsh}
Let ${d\nu = h d\gamma}$ be a centered probability measure on $\R^d$ with
smooth density $h$ with respect to $\gamma$.
Assume further that ${\rm S}(\nu \, | \, \gamma)$
and ${\rm H}(\nu \, | \, \gamma)$ are both positive and finite. Then
$$
{\rm W}_2(\nu,\gamma) \, \leq \,  {\rm S} \big (\nu \, | \, \gamma \big)\,
   {\rm arccos} \Big ( e^{-\frac{{\rm H}(\nu | \gamma)}{{\rm S}^2(\nu|\gamma)}} \Big ).
$$
\end{theorem}

\noindent
{\it Proof}.
For any $t\geq 0$, recall $d\nu^t =P_t h d\gamma$
(in particular, $\nu^0=\nu$ and $\nu^t \to\gamma$ as $t\to\infty$).
The HSI inequality \eqref{e:hsi2} applied to $\nu^t $ yields
that
$$
{\rm H} \big (\nu^t  \, | \, \gamma \big )
  \,\leq \,  \frac12 \,  {\rm S}^2 \big (\nu^t \, | \, \gamma \big ) \,
     \log\bigg(1+\frac{{\rm I}(\nu^t \, | \, \gamma)}{{\rm S}^2(\nu^t \, | \, \gamma)}\bigg).
$$
Now, ${\rm S}^2(\nu^t \, | \, \gamma)\leq {\rm S}^2(\nu \, | \, \gamma)$ by \eqref{e:decaystein}
and $r\mapsto r\log \big (1+ \frac {s}{r} \big )$ is increasing for any fixed $s$ from which it follows that
$$
{\rm H} \big (\nu^t \, | \, \gamma \big )
  \,\leq \,  \frac12 \,  {\rm S}^2 \big (\nu \, | \, \gamma \big ) \,
     \log\bigg(1+\frac{{\rm I}(\nu^t \, | \, \gamma)}{{\rm S}^2(\nu \, | \, \gamma)}\bigg).
$$
By exponentiating both sides, this inequality is equivalent to
$$
\sqrt{{\rm I} \big (\nu^t \, | \, \gamma \big)}
   \, \leq \,  \frac{{\rm I} (\nu^t \, | \, \gamma )}{{\rm S}(\nu \, | \, \gamma)
       \sqrt{e^{\frac{2{\rm H}(\nu^t|\gamma)}{{\rm S}^2(\nu|\gamma)}}-1}} \, .
$$
Combining with \eqref{ottovillani}  and recalling \eqref {e:bruijn} leads to
\begin {equation*} \begin {split}
\frac{d^+}{dt} \, {\rm W}_2(\nu,\nu^t)
   \, \leq \,  \sqrt{{\rm I} \big (\nu^t \, | \, \gamma \big) }
& \, \leq \,  -\frac{\frac{d}{dt}{\rm H}(\nu^t \, | \, \gamma)}{{\rm S}(\nu \, | \, \gamma)
     \sqrt{e^{\frac{2{\rm H}(\nu^t|\gamma)}{{\rm S}^2(\nu|\gamma)}}-1}}\\
& \, = \, -\frac{d}{dt}\bigg( {\rm S} \big (\nu \, | \, \gamma \big )\,{\rm arccos}\Big(
       e^{-\frac{H(\nu^t|\gamma)}{{\rm S}^2(\nu|\gamma)}} \Big)\bigg).
\end {split} \end {equation*}
In other words,
$$
\frac{d}{dt} \bigg( {\rm W} _2(\nu,\nu^t)
+ {\rm S} \big (\nu \, | \, \gamma)\,{\rm arccos}\Big( e^{-\frac{H(\nu^t|\gamma)}{S^2(\nu|\gamma)}}\Big)
    \bigg) \, \leq \,  0.
$$
The desired conclusion is achieved by integrating between $t=0$
and $t=\infty$. The proof of Theorem~\ref {t:wsh} is complete.
\qed

\bigskip

Proposition~\ref {p:ws} and Theorem~\ref {t:wsh} raise a number of observations.

\begin{remark}\label{r:r}

\begin{itemize}

\item[(a)] Since $\arccos(e^{-r}) \leq \sqrt{2r}$ for every $r\geq 0$, the WSH
inequality thus represents an improvement upon the Talagrand inequality
\eqref {e:talagrand}. Moreover, as for the HSI inequality, the WSH inequality
produces the case of equality in \eqref {e:talagrand} since
$\arccos(e^{-r}) \leq \sqrt{2r}$ is an equality only at $r = 0$.

\item[(b)] The Talagrand inequality may combined with the HSI inequality of Theorem~\ref{t:hsi}
to yield the bound
\begin {equation} \label {e:wsi}
{\rm W}_2^2 (\nu , \gamma ) \, \leq \,  {\rm S}^2 \big ( \nu \, | \, \gamma  \big )
   \log \bigg ( 1 + \frac{ {\rm I} ( \nu \, | \, \gamma  )}{{\rm S}^2  ( \nu \, | \, \gamma  ) }  \bigg )  .
\end {equation}
\item[(c)](HWI inequality). {As described in the introduction,
a fundamental estimate connecting entropy ${\rm H}$, Wassertein distance ${\rm W}_2$
and Fisher information ${\rm I}$ is the so-called HWI inequality
of Otto and Villani \cite{O-V} stating that, for all $d\nu = hd\gamma$ with density
$h$ with respect to $\gamma$,
\begin{equation}\label{e:hwi}
{\rm H} \big ( \nu \, | \, \gamma  \big ) \, \leq \,   {\rm W}_2 (\nu , \gamma )\,
    \sqrt{{\rm I} \big ( \nu \, | \, \gamma  \big )} - \frac{1}{2} \, {\rm W}_2 ^2(\nu , \gamma )
\end{equation}
(see, e.g.~\cite[pp. 529-542]{V} or \cite[Section~9.3.1]{B-G-L} for a general discussion).
Recall that the HWI inequality \eqref {e:hwi} improves upon both the logarithmic
Sobolev inequality \eqref {e:logsob} and the Talagrand inequality \eqref {e:talagrand}.}
It is natural to look for a more general inequality,
involving all four quantities ${\rm H}$, ${\rm W}_2 $, ${\rm I}$ and
the Stein discrepancy ${\rm S} $, and improving both the HSI and HWI inequalities.
One strategy towards this task would be to follow again the heat
flow approach of the proof of Theorem~\ref {t:hsi}
and write, for $0 < u \leq t $,
\begin{equation*} \begin {split}
  {\rm Ent}_\gamma (h)
     & \, = \,  \int_0^t {\rm I}_\gamma  (P_s h)   ds  + {\rm Ent}_\gamma (P_t h) \\
     &  \, \leq \,  {\rm I}_\gamma  (h) \int_0^u e^{- 2s} ds
          +  {\rm S}^2 \big ( \nu \, | \, \gamma  \big ) \int_u^t \frac{e^{-4s}}{1 - e^{-2s} } \, ds
        +  \frac{e^{- 2t}}{2(1 - e^{-2t})} \, {\rm W}_2^2 (\nu, \gamma) . \\
\end {split} \end {equation*}
Here, we used \eqref{e:decay} and \eqref{eq1.9}, {as well as the known
reverse Talagrand inequality along the semigroup given by}
$$  {\rm Ent}_\gamma (P_t h) \leq \frac{e^{- 2t}}{2(1 - e^{-2t})} \, {\rm W}_2^2 (\nu, \gamma)$$
(cf.~e.g.~\cite[p. 446]{B-G-L}). Setting $ \alpha = 1 - e^{-2u} \leq 1 - e^{-2t} = \beta$,
the preceding estimate yields
$$
{{\rm H} \big ( \nu \, | \,  \gamma  \big )
   \, \leq \,   \inf_{0 <  \alpha \leq   \beta \leq 1} \Phi (\alpha, \beta) }
$$
where
$$
\Phi (\alpha, \beta) \, = \,   \alpha \, {\rm I} \big ( \nu \, | \, \gamma  \big ) +  (\alpha - \log \alpha) \,
{\rm S}^2 \big ( \nu \, | \, \gamma  \big ) +  \frac {1- \beta}{\beta}
   \, {\rm W}_2^2(\nu, \gamma) +  (\log \beta  - \beta) \,
 {\rm S}^2 \big ( \nu \, | \, \gamma  \big ).
$$
However, elementary computations show that, unless the
rather unnatural inequality $2{\rm W}_2(\nu, \gamma)\leq {\rm S}  ( \nu \, | \, \gamma   )$
is verified, the minimum in the above expression is attained at a point $(\alpha , \beta)$ such
that either $ \alpha = \beta $ (and in this case one recovers HWI) or $\beta =1$ (yielding HSI).
Hence, at this stage, it seems difficult to outperform both HWI and HSI estimates with a
single `HWSI' inequality. In the subsequent point (d), we provide an elementary
explicit example in which
the HSI estimate perform better than the HWI inequality.

\item[(d)] {In this item, we thus compare the HWI and HSI inequalities on a specific
example in dimension $d = 1$.
For every $n\geq 1$, consider the probability measure $d\nu_n(x) = \rho_n(x)dx$
with density
$$
\rho_n(x) \, = \, \frac{1}{\sqrt{2\pi}}\big[(1-a_n)e^{-x^2/2}+na_ne^{- n^2x^2/2}\big], \quad x \in \R,
$$
where ${(a_n)}_{n \geq 1}$ is such that $a_n\in[0,1]$
for every $n \geq 1$, $a_n=o \big (\frac1{\log n} \big )$ and ${ n^{2/3} a_n\to\infty}$.
A direct computation easily shows that ${\rm H}(\nu_n \, | \, \gamma)\to 0$.
Also, since
$$
\rho'_n(x) \, = \,  - \frac{x}{\sqrt{2\pi}}\big[(1-a_n)e^{-x^2/2}+n^3a_ne^{-x^2n^2/2}\big],
$$
one may show after simple (but a bit lengthy) computations that
$$
{\rm I} \big (\nu_n \, | \, \gamma \big ) \, = \,  \int_\R \frac{\rho'_n(x)^2}{\rho_n(x)}dx - 1
 \, \sim \,  n^2a_n
\quad {\hbox {as}} \quad n \to \infty.
$$

We next examine the Stein discrepancy $ {\rm S}(\nu_n \, | \, \gamma)$
and Wassertein distance ${\rm W}_2(\nu_n,\gamma)$.
Since a Stein kernel $\tau_n$ of $\nu_n$ is given by
$$
\tau_n (x) \, = \, \frac{1}{\sqrt{2\pi} \, \rho_n}
     \Big[(1-a_n)e^{-x^2/2}+ \frac {a_n}{n} \, e^{- n^2x^2/2}\Big],
$$
it is easily seen that
$$
{\rm S}^2 \big (\nu_n \, | \, \gamma \big )
   \, = \,  \int_\R \big (\tau_n(x) - 1 \big )^2 \rho_n(x)dx  \, \leq \,   a_n \,  \to \,  0.
$$
Concerning the Wasserstein distance,
from the inequality \eqref {e:w2s},
we deduce that ${\rm W}_2(\nu_n,\gamma)\leq \sqrt{a_n}$. On the other hand,
by the Lipschitz characterization of ${\rm W_1}$ (specializing to the Lipschitz
function $ x \mapsto |\cos (x)|$), cf. e.g.~\cite[Remark 6.5]{V}),
$$
{\rm W}_2(\nu_n,\gamma) \, \geq \,  {\rm W}_1(\nu_n,\gamma)
     \, \geq \,   \bigg | \int_{\R} \big | \cos(x) \big | d\nu_n (x)
          - \int_{\R} \big | \cos(x) \big | d\gamma (x) \bigg | .
$$
Now, the right-hand side of this inequality multiplied by $\frac {1}{a_n}$ is equal to
\begin {equation*} \begin {split}
 \bigg | n \int_{\R} \big | \cos(x) \big |   e^{-n^2x^2/2} \,   \frac {dx}{\sqrt 2\pi}  - &
          \int_{\R} \big | \cos(x) \big | d\gamma (x) \bigg | \\
 & \, = \,   \bigg | \int_{\R} \Big [ \big | \cos ( {\textstyle \frac {x}{n}} ) \big |
          -   \big | \cos(x) \big | \Big ] d\gamma (x) \bigg |
\end {split} \end {equation*}
which, by dominated convergence, converges to a non-zero limit.
As a consequence, there exists $c>0$ such that, for $n$ large enough,
$ {\rm W}_2(\nu_n,\gamma)\geq c\, a_n$.

Summarizing the conclusions, the quantity
$$
{\rm W}_2 \big (\nu_n,\gamma \big )
      \sqrt{{\rm I} \big (\nu_n \, | \, \gamma \big )}
            - \frac12 \, {\rm W}_2 ^2 (\nu_n,\gamma)
$$
is bigger than a sequence of the order of $na_n^{3/2}=(n^{2/3}a_n)^{3/2}$, which
(by construction) diverges to infinity as $n\to\infty$.
This fact implies that, in this specific case, the bound in the HWI inequality diverges
to infinity, whereas ${\rm H}(\nu_n \, | \, \gamma)\to 0$.
On the other hand, the HSI bound converges to zero, since
$$
  {\rm S}^2 \big (\nu_n \, | \, \gamma \big)
   \log\bigg (1+\frac{{\rm I}(\nu_n \, | \, \gamma)}{{\rm S}^2(\nu_n \, | \, \gamma)}\bigg)
     \,\leq \,  a_n\log(1+n^2) \, \sim \,  2a_n\log n \, \to \,  0.
$$
}
\end{itemize}
\end{remark}

\subsection{General Wasserstein distances under a stronger notion of Stein kernel}\label{ss:W3}

In this part, we obtain bounds in terms of Stein discrepancies on the
Wasserstein distance ${\rm W}_p$ of any order $p$ between a centered probability measure $\nu$
on $\R^d$ and the standard Gaussian distribution $\gamma$.
As in Proposition~\ref{p:ws}, we shall consider probabilities $\nu$ not
necessarily admitting a density with respect to $\gamma$. However, it
will be assumed that $\nu$ has a Stein kernel $\tau_\nu$ verifying the
stronger `vector' relation \eqref{e:strongersm}. The reason for
this is that, in order to deal with Wasserstein distances of the type ${\rm W}_p$, $p\neq 2$,
one needs to have access to the explicit expression of the score function
$\nabla (\log P_t h)$ along the Ornstein-Uhlenbeck semigroup,
as proved in \cite[Lemma 2.9]{N-P-S} in the framework of Stein kernels
verifying \eqref{e:strongersm}. Recall that the existence
of $\tau_\nu$ implies that $\nu$ has finite moments of order $2$.

\begin {proposition} [${\rm W}_p$ distance and Stein discrepancy] \label{p:ws-2}
Let $\nu$ be a centered probability measure on $\R^d$ with Stein kernel $\tau_\nu$
in the sense of \eqref{e:strongersm}. For every $p\geq 1$, set
$$
{\| \tau_\nu - {\rm Id} \|}_{p,\nu }
     \, = \, \bigg (\sum_{i,j=1}^d \int_{\R^d} \big | \tau_\nu ^{ij}
        - \delta_{ij} \big |^p d\nu \bigg )^{1/p}
$$
(where $\delta_{ij} = 1$ if $ i=j$ and $0$ if not), possibly infinite
if $\tau_\nu^{ij} \notin {\rm L}^p(\nu)$.
In particular, $ {\| \tau_\nu - {\rm Id} \|}_{2,\nu } = {\rm S}(\nu \, | \, \gamma)$.
\begin{itemize}
\item[(i)] Let $p\in [1,2)$. Then,
\begin{equation}\label{e:w1-2}
{\rm W}_p(\nu, \gamma) \, \leq \,  C_p \, d^{1-1/p}  {\| \tau_\nu - {\rm Id} \|}_{p,\nu }
\end{equation}
where $C^p_p = \int_{\R} |x|^p d \gamma^1 (x)$.

\item[(ii)] Let $p\in [2, \infty)$. If $\nu$ has finite moments of order $p$, then
(with the same $C_p$ as in  (i))
\begin{equation}\label{e:w2-2}
{\rm W}_p(\nu, \gamma) \, \leq \,  C_p \, d^{1-2/p}  \, {\| \tau_\nu - {\rm Id} \|}_{p,\nu }.
\end{equation}
In particular, for $p=2$ we recover \eqref{e:w2s}.
\end{itemize}
\end {proposition}

\begin{proof}
Owing to an approximation argument analogous to the one rehearsed
 at end of the proof of Proposition \ref{p:ws}, it is sufficient to consider the case
 $d\nu = h\,d\gamma$ where $h$ is a smooth density. Write as before
$v_t = \log P_t h$ and $d\nu^t = P_th d\gamma$.
By virtue of \cite[Lemma 2.9]{N-P-S}, under thus the strengthened
assumption \eqref{e:strongersm}, a version of $\nabla v_t$, $t>0$, is given by
$$
x \, \mapsto \, \nabla v_t(x) \, = \,
  \frac{e^{-2t}}{\sqrt{1-e^{-2t}}} \, \E \big [ \big (\tau_\nu(F) - {\rm Id} \big) Z \, |\, F_t = x \big ],
      \quad x\in \R^d,
$$
where, as in Remark~\ref {r:conditioning},
$F$ and $Z$ are independent with respective law $\nu$ and $\gamma$,
and $F_t = e^{-t}F+\sqrt{1-e^{-2t}} Z$. Moreover, one can
straightforwardly modify the proof of \cite[Lemma~2]{O-V} (cf.~also
\cite[Theorem~24.2(iv)]{V}) in order to obtain the general estimate
\begin{equation}\label{ottovillani-2}
 \frac{d^+}{dt} \,  {\rm W}_p(\nu,\nu^t)
 \, \leq \,  \left( \int_{\R^d} \left|\nabla v_t  \right|^p d\nu^t \right)^{1/p}.
\end{equation} It follows that
\begin {equation*} \begin {split}
{\rm W}_p(\nu,\gamma)
& \, \leq \,  \int_0^\infty \left(\int_{\R^d} | \nabla v_t|^pd \nu^t\right)^{1/p}dt\\
& \, = \, \int_0^\infty \! \frac{e^{-2t}}{\sqrt{1-e^{-2t}}} \,
\E\bigg [\bigg( \sum_{i=1}^d \E\bigg[
\sum_{j=1}^d \big ( \tau_\nu^{ij}(F) - \delta_{ij} \big )Z_j \bigg| F_t
\bigg]^2 \bigg)^{p/2}\bigg]^{1/p}dt.\\
\end {split} \end {equation*}
Now, if $1\leq p< 2$,
\begin {equation*} \begin {split}
{\rm W}_p(\nu,\gamma)
& \, \leq \,   \int_0^\infty \! \frac{e^{-2t}}{\sqrt{1-e^{-2t}}} \, dt  \,
\bigg(\sum_{i=1}^d \E\bigg[
\bigg| \sum_{j=1}^d \big ( \tau_\nu ^{ij} (F) - \delta_{ij} \big )Z_j
\bigg|^{p}\bigg]\bigg)^{1/p} \\
& \, \leq \,  C_p \, d^{1-1/p}
\bigg(\sum_{i,j=1}^d \E\big [ \big | \tau_\nu^{ij}(F) - \delta_{ij} \big |^p \big ]\bigg)^{1/p}
\end {split} \end {equation*}
yielding {\it (i)}.
On the other hand, if $p\geq 2$, then
\begin {equation*} \begin {split}
{\rm W}_p(\nu,\gamma)
& \, \leq \,  \int_0^\infty \! \frac{e^{-2t}}{\sqrt{1-e^{-2t}}} \,
\, \E\bigg[\bigg( \sum_{i=1}^d \E \bigg[\bigg(
\sum_{j=1}^d \big ( \tau_\nu^{ij}(F) - \delta_{ij} \big )Z_j \bigg )^2 \bigg | F_t \bigg]
 \bigg)^{p/2} \bigg]^{1/p}dt\\
& \, \leq \,  d^{1/2-1/p} \bigg(\sum_{i=1}^d \E \bigg [ \bigg |
\sum_{j=1}^d \big  ( \tau_\nu^{ij}(F) - \delta_{ij}\big )Z_j \bigg |^p \bigg] \bigg )^{1/p}\\
& \, = \,  C_p \, d^{1/2-1/p} \bigg(\sum_{i=1}^d  \bigg(
\sum_{j=1}^d \E \Big [ \big ( \tau_\nu^{ij}(F) - \delta_{ij}\big )^2 \Big]
 \bigg )^{p/2} \bigg)^{1/p}\\
& \, \leq \,  C_p \, d^{1-2/p} \bigg (\sum_{i,j=1}^d
 \E\big [ \big | \tau_\nu ^{ij} (F) - \delta_{ij} \big |^p \big ] \bigg )^{1/p}
\end {split} \end {equation*}
which immediately yields {\it (ii)}. The proof of Proposition~\ref {p:ws-2} is complete.
\end{proof}

\begin{remark}
Specializing \eqref{e:w1-2}
to the case $p=1$ yields the estimate
\begin {equation} \label{e:ws1}
 {\rm W}_1(\nu,\gamma)  \, \leq \,  \sqrt { \frac{2}{\pi} } \,
     {\| \tau_\nu - {\rm Id} \|}_{1,\nu }
\end {equation}
which improves previous dimensional bounds obtained
by an application of the multidimensional Stein method
(cf.~the proof of \cite[Theorem~6.1.1]{N-P-12}).
It is important to note that, apart from the results obtained
in the present paper, there is no other version of Stein's method
allowing one to deal with Wasserstein distances of order $p>1$.
Observe that coupling results from \cite{C3} (that are based on completely different methods)
may be used to deduce analogous estimates in the case when $d=1$ and the Stein kernel
$\tau_\nu$ is bounded.
\end{remark}

\section{HSI inequalities for further distributions}\label{S4}

On the basis of the Gaussian example of Section~\ref {S2}, we next address the issue of
HSI inequalities for distributions on $\R^d$, $d\geq 1$, that are not necessarily Gaussian.
In order to reach the basic semigroup ingredients towards such HSI inequalities put forward in
Proposition~\ref {p:up}, a convenient family of measures to deal with
is the family of invariant measures of second order differential operators.
These include gamma and beta distributions, as well as families of log-concave
measures as illustrations. As such, the investigation is part of the
generator approach to Stein's method as developed in \cite{Ba, G,R}.
We present it here in the framework of Markov Triples
as developed in \cite{B-G-L} and, for simplicity, only consider operators
and measures on $\R^d$.

\subsection{A general statement}\label{S41}

Let $E$ be a domain of $\R^d$ and consider a family of real-valued
$C^\infty$-functions $a^{ij}(x)$ and $b^i(x)$, $i, j = 1, \ldots, d$, defined on $E$.
We assume that the matrix $a(x)={(a^{ij}(x))}_{1\leq i,j\leq d}$ is symmetric
and positive definite for any $x\in E$. For every $x\in E$, we let
$a^{\frac12}(x)$ be the unique symmetric non-singular matrix such that
$(a^{\frac12}(x))^2= a(x)$. Let $\mathcal{A}$ denote the algebra of $C^\infty$-functions on $E$
and $\mathcal{L}$ be the second order differential operator given on functions
$ f\in\mathcal{A}$ by
\begin{equation}  \label{2nd}
\mathcal{L}f
   \, = \, {\big \langle a, {\rm Hess}(f) \big \rangle}_{\rm HS}  + b \cdot \nabla f
    \, = \, \sum_{i,j=1}^d a^{ij} \, \frac{\partial^2 f}{\partial x_i\partial x_j}
    +\sum_{i=1}^d b^i \, \frac{\partial f}{\partial x_i} \, .
\end{equation}
The operator $\mathcal {L}$ satisfies the chain rule formula
and defines a diffusion operator.
We assume that $\mathcal{L}$ is the generator of a symmetric Markov semigroup
${(P_t)}_{t\geq 0}$, where the symmetry is with respect to an invariant probability measure $\mu$.

A central object of interest in this context is the
carr\'e du champ operator $\Gamma $ defined from the generator $\mathcal{L}$ by
$$
\Gamma (f,g) \, = \,  \frac12\big[ \mathcal{L}(fg) - f \mathcal{L}g - g  \mathcal{L}f \big ]
     = \sum_{i,j=1}^d a^{ij} \, \frac{\partial f}{\partial x_i}\frac{\partial g}{\partial x_j}
$$
for all $(f,g) \in \mathcal{A} \times \mathcal{A}$. Note that $\Gamma$ is bilinear
and symmetric and $\Gamma (f,f) \geq 0$. Moreover, the integration by parts property
for $ \mathcal {L}$ with respect to the invariant measure
$\mu$ is expressed by the fact that, for functions $f, g \in \mathcal {A}$,
$$
\int_E f \, \mathcal {L} g \, d\mu  \, = \,  - \int _E \Gamma (f,g) d\mu.
$$
The structure $(E, \mu, \Gamma)$ then defines a Markov Triple in the sense of \cite{B-G-L}
to which we refer for the necessary background.

{The requested semigroup analysis toward HSI inequalities}
will actually involve in addition the iterated gradient operators
$\Gamma_n$, $n\geq 1$, defined inductively for $(f,g) \in \mathcal{A} \times \mathcal{A}$
via the relations $\Gamma_0(f,g)=fg$ and
$$
\Gamma_n(f,g) \, = \, \frac12\big[ \mathcal{L} \, \Gamma_{n-1}(f,g)
    -\Gamma_{n-1}(f,\mathcal{L}g)-\Gamma_{n-1}(g,\mathcal{L}f)\big], \quad n\geq 1.
$$
In particular $\Gamma_1 = \Gamma$ and the operators $\Gamma_n$, $n \geq 1$,
are similarly symmetric and bilinear.
In what follows, we shall often adopt the
shorthand notation $\Gamma_n(f)$ instead of $\Gamma_n(f,f)$.
The $\Gamma_2$ operator is part of the famous Bakry-\'Emery criterion for logarithmic Sobolev
inequalities \cite {B-E}, \cite[Section~5.7] {B-G-L}.
As a new feature of the analysis here, the iterated gradient $\Gamma_3$ will turn
essential towards a suitable analogue of {\it (iii)} in  Proposition~\ref {p:up}.

{A prototypical example of this setting is of course the Ornstein-Uhlenbeck
operator $\mathcal {L} = \Delta - x \cdot \nabla $ on $\R^d$ considered earlier, with the standard
Gaussian measure $\gamma $ as symmetric and invariant measure. In this case, the
carr\'e du champ operator is simply given by $\Gamma (f) = |\nabla f|^2$
on smooth functions $f$. It is easily seen that, for example (cf.~\cite{L1}),
$$
\Gamma_2(f) \, = \,  \sum_{i,j = 1}^d  \bigg(\frac{\partial^2 f}{\partial x_i\partial x_j} \bigg)^2
     + \Gamma (f)
$$
and
$$
\Gamma_3(f) \, = \,   \sum_{i,j, k=1}^d \!
\bigg(\frac{\partial^3 f}{\partial x_i\partial x_j\partial x_k} \bigg)^2
    + 3\, \Gamma_2(f) - 2 \, \Gamma (f).
$$
}

\medskip

Given thus the preceding Markov Triple $(E, \mu, \Gamma)$ associated to
the second order differential operator $\mathcal {L}$ of \eqref {2nd},
let $d\nu = h d\mu$ where $h$ is a smooth probability density with respect to $\mu$.
As in the Gaussian case, the
relative entropy of $\nu$ with respect to $\mu$ is the quantity
$$
{\rm H} \big (\nu\, |\, \mu \big ) \, = \,  {\rm Ent}_\mu(h)=\int_E h \log h \, d\mu.
$$
Similarly, the Fisher information of $\nu$ (or $h$) with respect to $\mu$ is defined as
\begin {equation} \label {e:fisher}
{\rm I} \big (\nu\,|\, \mu \big )
    \, = \, {\rm I}_\mu(h) \, = \, \int_E \frac{\Gamma(h)}{h} \,d\mu
    \, = \, \int_E \Gamma(\log h) h d\mu  \, = \,  -\int_E \mathcal{L}(\log h) d\nu.
\end {equation}
The (integrated) {\em de Bruijn's identity} (cf.~Proposition 5.2.2 in \cite{B-G-L})
reads as in {\it (i)} of Proposition~\ref {p:up},
$$
{\rm H} \big (\nu\, |\, \mu \big ) \, = \,  \int_0^\infty {\rm I}_\mu(P_th)d\mu.
$$

Let $\mathcal{M}_{d\times d}$ denote the class of $d\times d$ matrices
with real entries. Analogously to the definition of Stein kernel of Section~\ref {S2.1},
we shall say that a matrix-valued mapping $\tau_\nu : \R^d \to \mathcal{M}_{d\times d}$
satisfying $\tau^{ij}_\nu\in {\rm L}^1(\nu)$ for every $i,j = 1, \ldots, d$, and
\begin{equation}\label {e:steinmatrix2}
- \int_E b \cdot \nabla f \, d\nu
    \, = \, \int_E {\big \langle \tau_\nu , {\rm Hess} (f) \big \rangle }_{\rm HS} \, d\nu,
      \quad f\in\mathcal{A},
\end{equation}
is a Stein kernel for the probability $\nu $ on $E$ with respect to
the generator of $\cal L$ of \eqref{2nd}, where $b = {(b_i(x))}_{1 \leq i \leq d}$
is part of the definition of $\cal L$.
For the Ornstein-Uhlenbeck operator ${\cal L} = \Delta - x \cdot \nabla$,
the definition corresponds to \eqref {e:steinmatrix}.
Since $\int_E \mathcal{L}f\,d\mu=0$, observe that $a$ is a Stein kernel for $\mu$.
The main result in this section is an HSI inequality
that relates ${\rm H}(\nu\, |\, \mu)$, ${\rm I}(\nu\, |\, \mu)$ and
the Stein discrepancy of $\nu$ with respect to $\mu$
\begin{equation}\label{e:gs2}
{\rm S} \big (\nu\, |\, \mu \big ) \, = \,   \bigg ( \int_E \big \|a^{-\frac12}\tau_\nu a^{-\frac12}
     - {\rm Id} \big \|^2_{\rm HS}\,d\nu \bigg)^{1/2}
\end{equation}
that we regard, as in the Gaussian case of Section~\ref {S2},
as a measure of the distance between $\nu$ and $\mu$
(since $\tau_\mu = a$).
Note that choosing $a=C$ in \eqref{e:gs2}, with $C$ non-singular, yields the quantity arising in
Corollary~\ref {c:hsi}. It should also be mentioned that
the Stein discrepancy \eqref{e:gs2} is somewhat in contrast with the bounds one customarily
obtains when applying Stein's method (see e.g.~\cite{N-P-09} for the specific example
of the one-dimensional Gamma distribution, or \cite{R} for a general reference), which
typically involve quantities of the type $\int_E \| \tau_\nu -a\|^2_{\rm HS}\,d\nu$.
The appearance of the inverse matrices $a^{-\frac12}$ seems to be inextricably
connected with the fact that we deal with information-theoretical functionals.

\medskip

The following general statement collects the necessary assumptions on the
iterated gradients $\Gamma$, $\Gamma_2$ and $\Gamma_3$ to achieve
the expected HSI inequality by the semigroup interpolation scheme.
The next paragraphs will provide illustrations in various concrete instances of interest.
In Theorem~\ref {t:abstract} below, {\it (i)} amounts to the Bakry-\'Emery
$\Gamma_2$ criterion
to ensure the logarithmic Sobolev inequality
(cf.~\cite[Section~5.7]{B-G-L}) while condition {\it (ii)} linking the
$\Gamma_2$ and $\Gamma_3$ operators will provide (together with {\it (iii)})
the suitable semigroup bound for the time control of
${\rm I}(P_th)$ away from $0$.
Recall $\Psi (r) = 1 + \log r $ if $r \geq 1$ and $ \Psi (r) = r $ if $0 \leq r \leq 1$.

\begin{theorem} [General HSI inequality] \label{t:abstract}
In the preceding context, let $d\nu = h d\mu$
where $h$ is a smooth density with Stein kernel $\tau_\nu$ with respect to $\mu$.
Assume that there exists $\rho,\kappa,\sigma >0$ such that,
for any $f\in\mathcal{A}$,
\begin{enumerate}
\item[(i)] $\Gamma_2(f)\ge \rho\,\Gamma (f)$;
\item[(ii)] $\Gamma_3(f) \ge \kappa\, \Gamma_2(f)$;
\item[(iii)] $\Gamma_2(f)\ge \sigma\,\|a^\frac12\,{\rm Hess}(f)\,a^\frac12\|_{\rm HS}^2$
(with $a$ as in \eqref{2nd}).
\end{enumerate}
Then,
$$
{\rm H}(\nu\, |\, \mu) \, \leq \,  \frac{1}{2\sigma } \, {\rm S}^2 \big (\nu\, |\, \mu\big ) \,
   \Psi \bigg ( \frac{\sigma \max ( \rho, \kappa) \, {\rm I}(\nu\, |\, \mu)}{\rho \kappa \, {\rm S}^2(\nu\, |\, \mu)} \bigg) .
$$
\end {theorem}

{Note that in the Ornstein-Uhlenbeck example, $\rho = \kappa = \sigma =1$
from which we recover the HSI inequality \eqref {e:hsi2}, however
in a slightly weaker formulation.}

\begin {proof}
It is therefore a classical fact (see e.g.~\cite[(5.7.4)]{B-G-L}) that
{\it (i)} ensures the exponential
decay of the Fisher information along the semigroup
\begin{equation}\label{i2}
{\rm I}_\mu(P_t h) \, \leq \,  e^{-2\rho t} \, {\rm I}_\mu(h)
 \, = \,  e^{-2\rho t} \, {\rm I} \big (\mu \, | \, \nu \big )
\end{equation}
for every $t\geq 0$ {(and then yields a logarithmic Sobolev inequality for $\mu$.)}
Now, fix $t>0$ and let $f \in \mathcal{A}$. {The $\Gamma$-calculus
as developed in \cite {B-G-L}, but at the level of the $\Gamma_2$ and
$\Gamma_3$ operators, yields on $[0,t]$ (by the very definition of $\Gamma_3$
from $\Gamma_2$),}
\begin{equation*} \begin {split}
\frac{d}{ds} \Big ( P_s \big (\Gamma_2(P_{t-s}f)\big )e^{-2\kappa s} \Big)
 & \, = \,  2e^{-2\kappa s}\Big (P_s \big (\Gamma_3(P_{t-s}f) \big )
     - \kappa P_s \big (\Gamma_2(P_{t-s}f)\big )\Big )\\
 & \, = \,  2e^{-2\kappa s}P_s \big ((\Gamma_3-\kappa \, \Gamma_2)(P_{t-s}f) \big ). \\
\end {split} \end {equation*}
By {\it (ii)}, the latter is non-negative so that the map
$s\mapsto P_s(\Gamma_2(P_{t-s}f))e^{-2\kappa s}$ is increasing on $[0,t]$, and thus
\begin{equation*} \begin {split}
P_t \big (\Gamma(f) \big ) - \Gamma \big (P_t(f) \big)
     & \, = \,  2\int_0^t P_s \big (\Gamma_2(P_{t-s}f) \big )ds \\
     & \, \geq  \, 2\,\Gamma_2(P_tf) \int_0^t e^{2\kappa s}ds
   \, = \,   \frac1\kappa \, (e^{2\kappa t}-1) \, \Gamma_2(P_tf).
\end {split} \end {equation*}
Together with {\it (iii)}, it then follows that
\begin{equation}\label{abc}
P_t \big (\Gamma (f) \big )
    \, \geq \,  P_t \big (\Gamma (f) \big )-\Gamma \big (P_t(f) \big )
       \, \geq \,  \frac{\sigma}{\kappa} \,  (e^{2\kappa t}-1 ) \,
       \big \|a^\frac12\,{\rm Hess}(P_tf)a^\frac12 \big \|_{\rm HS}^2 .
\end{equation}

We shall apply \eqref {abc} to $v_t = \log P_t h$
(with $h$ regular enough). First, by symmetry of $\mu$ with respect to ${(P_t)}_{t\geq 0}$,
\begin {equation} \label {e:information}
{\rm I}_\mu(P_th) \, = \,  -\int_E \mathcal{L}v_t\,P_t h\,d\mu
  \, = \,  -\int_E \mathcal{L}P_tv_t \, hd\mu \, = \,  -\int_E \mathcal{L} P_tv_t \, d\nu.
\end {equation}
Hence, by \eqref {2nd} and \eqref {e:steinmatrix2},
\begin{equation*} \begin {split}
{\rm I}_\mu(P_th)
& \, = \, - \int_E {\big \langle a, {\rm Hess}(P_t v_t) \big \rangle }_{\rm HS} \, d\nu
                  - \int_E b \cdot \nabla P_t v_t \, d\nu \\
& \, = \,  \int_E \big \langle \tau_\nu -a,{\rm Hess}(P_t v_t) \big \rangle_{\rm HS} \,d\nu. \\
\end {split} \end {equation*}
Now, by the Cauchy-Schwarz inequality,
\begin{equation*} \begin {split}
{\rm I}_\mu(P_th) & \, = \,
\int_E \big \langle a^{-\frac12}\tau_\nu a^{-\frac12}- {\rm Id} ,a^{\frac12}\,
    {\rm Hess}(P_t v_t) a^{\frac12} \big \rangle_{\rm HS} \, d\nu \\
& \, \leq \,  \bigg (\int_E \big \|a^{-\frac12}\tau_\nu a^{-\frac12}
    - {\rm Id} \big \|^2_{\rm HS} \, d\nu \bigg)^{1/2}
  \bigg ( \int_E \big \|a^{\frac12}\,{\rm Hess}(P_t v_t) a^{\frac12} \big \|^2_{\rm HS}
       \, d\nu \bigg)^{1/2} \\
& \, \leq \, {\rm S} \big (\nu\,|\,\mu \big ) \bigg ( \frac{\kappa}{\sigma (e^{2\kappa t}-1)}
     \int_E P_t \big (\Gamma (v_t)) \big ) d\nu \bigg)^{1/2}
\end {split} \end {equation*}
where the last step follows from \eqref {abc}.
Since
$$
 \int_E P_t \big (\Gamma (v_t) \big ) d\nu \, = \,
    \int_E P_t \big (\Gamma (v_t) \big )h d\mu  \, = \,  \int_E \Gamma (v_t)P_t hd\mu
     \, = \, {\rm I}_\mu(P_th) ,
$$
it follows that
\begin{equation}\label{e:i1}
{\rm I}_\mu(P_th) \, \leq \,  \frac{\kappa}{\sigma(e^{2\kappa t}-1)} \, {\rm S}^2 \big (\nu\,|\,\mu \big ).
\end{equation}

Finally, using \eqref{i2} for small $t$ and \eqref{e:i1} for large $t$, one deduces that, for every $u>0$,
\begin{equation*} \begin {split}
{\rm H} \big (\nu\,|\,\mu \big)
  & \, \leq \,  {\rm I} \big (\nu\,|\,\mu \big )\int_0^u e^{-2\rho t}dt +
     {\rm S}^2 \big (\nu\,|\,\mu \big )\int_u^\infty \frac{\kappa}{\sigma (e^{2\kappa t}-1)} \,dt\\
 & \, = \, \frac{{\rm I} (\nu\,|\,\mu)}{2\rho} \, (1-e^{-2\rho u})
  - \frac{{\rm S}^2 (\nu\,|\,\mu)}{2\sigma} \, \log(1-e^{-2\kappa u}).
\end {split} \end {equation*}
Setting $r =e^{-2u}$,
$$
{\rm H} \big (\nu\, |\, \mu \big )
\, \leq \,  \inf_{0<r<1}\bigg \{\frac{{\rm I}(\nu\, |\, \mu)}{2\rho}\,(1-r^\rho)
    -\frac{{\rm S}^2(\nu\, |\, \mu)}{2\sigma} \, \log(1- r^\kappa)\bigg \}.
$$
{Now, using that $1- r^{\rho} \leq \max (1, \frac {\rho}{\kappa}) (1- r^\kappa)$
for $ r \in (0,1)$, a simple (non-optimal) optimization yields the
desired conclusion. The proof of Theorem~\ref {t:abstract}
is complete.}
\end {proof}

\begin {remark} \label {r:abstractw}
It should be pointed out that, on the basis of \eqref {e:i1}, transport inequalities
as studied in Section~\ref {S3} may be investigated similarly in the preceding
general context, and with similar illustrations as developed below. For example,
as an analogue of \eqref {e:w2s},
$$
{\rm W}_2 (\nu, \mu) \, \leq \, \frac {2}{\sqrt {\kappa \sigma}} \,  {\rm S} \big (\nu\, |\, \mu \big ).
$$
In order not to expand too much the exposition, we leave the details to the reader.
\end {remark}

The next paragraphs present various illustrations of Theorem~\ref {t:abstract}.

\subsection{Multivariate gamma distribution}  \label{gamma-sec}

{As a first example of illustration of the preceding general result,} we
consider the case of the multidimensional Laguerre operator, which is the product on $\R_+^d$ of one-dimensional Laguerre operators of parameters $p_i >0$, $i = 1, \ldots, d$, that is,
$$
\mathcal{L}f \, = \,  \sum_{i=1}^d x_i \,\frac{\partial^2 f}{\partial x_i^2} + \sum_{i=1}^d (p_i - x_i)
 \,\frac{\partial f}{\partial x_i} \, .
$$
In particular, $a(x) = {(x_i \delta_{ij})}_{1 \leq i, j \leq d}\, $ in \eqref{2nd}.
It is a standard fact that the invariant measure $\mu$ associated with $\mathcal{L}$
has a density with respect to the Lebesgue measure given by the tensor product
of $d$ gamma densities of the type $\Gamma(p_i)^{-1} x_i^{p_i-1}e^{-x_i}$,
$x_i \in \R_+$, $i=1,\ldots ,d$. For reasons that will become clear later on, we assume that
$p_i \geq \frac {3}{2}$, $ i = 1, \ldots , d$.

After some easy but cumbersome calculations, it may be checked that, along suitable smooth functions $f$,
\begin{eqnarray*}
\Gamma (f) &=& \sum_{i=1}^d x_i \left(\frac{\partial f}{\partial x_i}\right)^2 \\
\Gamma_2(f) &=& \sum_{i,j=1}^d x_i x_j \left(\frac{\partial^2 f}{\partial x_i\partial x_j} \right)^2
    + \sum_{i=1}^d x_i \, \frac{\partial f}{\partial x_i} \,  \frac{\partial^2 f}{\partial x_i^2}
    + \frac {1}{2} \sum_{i=1}^d (p_i+x_i) \left( \frac{\partial f}{\partial x_i}\right)^2 \\
 \Gamma _3(f) & =& \sum_{i,j, k=1}^d x_i x_j x_k \left(\frac{\partial^3 f}{\partial x_i\partial x_j\partial x_k} \right)^2
     + 3 \sum_{i,j=1}^d x_i x_j  \, \frac{\partial^2 f}{\partial x_i\partial x_j} \, \frac{\partial^3 f}{\partial x_i^2\partial x_j} \\
     &&  + \frac {3}{2} \sum_{i,j=1}^d (p_i+x_i) x_j\left(\frac{\partial^2 f}{\partial x_i\partial x_j} \right)^2
       +  \frac {3}{2}  \sum_{i=1}^d x_i \left(\frac{\partial^2 f}{\partial x_i^2}\right)^2   \\
&&     + \frac {3}{2} \sum_{i=1}^d x_i \, \frac{\partial f}{\partial x_i} \,  \frac{\partial^2 f}{\partial x_i^2}
       +  \frac {1}{4}  \sum_{i=1}^d (3p_i +x_i) \left(\frac{\partial f}{\partial x_i}\right)^2.
\end{eqnarray*}
Note that (recall $x_i, x_j , x_k \geq 0$)
\begin{equation*} \begin {split}
\sum_{i,j, k=1}^d x_i x_j  x_k &  \left(\frac{\partial^3 f}{\partial x_i\partial x_j \, \partial x_k} \right)^2
     + 3 \sum_{i,j=1}^d x_i x_j \frac{\partial^2 f}{\partial x_i\partial x_j} \, \frac{\partial^3 f}{\partial  x_i^2\partial x_j}\\
     & \, \geq \,  \sum_{i,j=1}^d x_i ^2 x_j  \left(\frac{\partial^3 f}{\partial x_i^2\partial x_j} \right)^2
      + 3 \sum_{i,j=1}^d x_i x_j \, \frac{\partial^2 f}{\partial x_i\partial x_j} \,
      \frac{\partial^3 f}{\partial x_i^2\partial x_j} \\
      & \, \geq \,   - \frac {9}{4}  \sum_{i,j=1}^d  x_j \left(\frac{\partial^2 f}{\partial x_i\partial x_j} \right)^2.
\end {split} \end {equation*}
Therefore
\begin{equation*} \begin {split}
 \Gamma _3(f) & \, \geq \, \frac {3}{2} \sum_{i,j=1}^d \big (p_i-\frac32+x_i \Big) x_j
      \left(\frac{\partial^2 f}{\partial x_i\partial x_j} \right)^2 +
   \frac {3}{2}  \sum_{i=1}^d x_i \left(\frac{\partial^2 f}{\partial x_i^2}\right)^2 \\
&  \, \quad     + \frac {3}{2} \sum_{i=1}^d x_i \, \frac{\partial f}{\partial x_i} \,
     \frac{\partial^2 f}{\partial x_i^2}
       +  \frac {1}{4}  \sum_{i=1}^d (3p_i +x_i) \left(\frac{\partial f}{\partial x_i}\right)^2.
\end {split} \end {equation*}
Since  $p_i \geq \frac {3}{2}$, it follows at once that
$ \Gamma_3 (f) \geq \frac {1}{2} \, \Gamma_2 (f) $.
Analogous computations lead to
$$
 \sum_{i=1}^d x_i \, \frac{\partial f}{\partial x_i} \,  \frac{\partial^2 f}{\partial x_i^2}
 \, \geq \,  -\frac12  \sum_{i=1}^d x_i^2 \left(\frac{\partial^2 f}{\partial x_i^2} \right)^2 -
 \frac {1}{2} \sum_{i=1}^d \left( \frac{\partial f}{\partial x_i}\right)^2,
$$
implying that
$$
\Gamma_2(f) \, \geq \,   \frac12 \sum_{i=1}^d x_i^2 \left(\frac{\partial^2 f}{\partial x_i^2} \right)^2
    + \frac {1}{2} \sum_{i=1}^d (p_i-1+x_i) \left( \frac{\partial f}{\partial x_i}\right)^2
 \,  \geq \,   \frac {1}{2} \, \Gamma (f).
$$
Finally, one has
\begin{equation*} \begin {split}
\frac12  \, \big \|\sqrt{a}\,{\rm Hess}(f)\,\sqrt{a} \big \|_{\rm HS}^2
   & \, = \,  \frac12\sum_{i,j=1}^d x_ix_j\left(\frac{\partial^2 f}{\partial x_i\partial x_j} \right)^2\\
  & \, \leq \,  \frac12\sum_{i,j=1}^d x_ix_j\left(\frac{\partial^2 f}{\partial x_i\partial x_j} \right)^2
+\frac12\sum_{i=1}^d \left( x_i \, \frac{\partial^2 f}{\partial x_i^2} + \frac{\partial f}{\partial x_i}\right)^2 \\
   & \, \leq \,  \Gamma_2(f).
\end {split} \end {equation*}

As a consequence, Theorem \ref{t:abstract} applies with
$\rho=\kappa=\sigma=\frac12$ to yield the following result
(the numerical constants there are not sharp).
The restrictions $p_i \geq \frac {3}{2}$, ${ i = 1, \ldots , d}$, are probably not optimal.
For example, it is not difficult to see from the preceding computations that
in the one-dimensional
case $d=1$, it is actually enough to assume that $p \geq \frac12$.

\begin {proposition} [HSI inequality for gamma distribution] \label {p:gamma}
Let $\mu$ be the product measure of gamma distributions
$\Gamma (p_i)^{-1} x_i^{p_i-1}e^{-x_i} dx_i$ on
$\R_+^d$ with $p_i \geq \frac {3}{2}$, $ i = 1, \ldots , d$.
Then, for any $d\nu = h d\mu$ where $h$ is a smooth probability density,
$$
{\rm H} \big (\nu\,|\,\mu \big )
    \, \leq \,  {\rm S}^2 \big (\nu\, |\, \mu \big ) \,
      \Psi \bigg ( \frac{{\rm I}(\nu\,|\,\mu)}{{\rm S^2}(\nu\,|\,\mu)}\bigg) .
$$
\end {proposition}

\subsection{One-dimensional uniform distribution on $[-1,+1]$}

{In this section, we examine the case of the one-dimensional Jacobi operator of parameters
$\alpha=\beta=1$, that is,
$$
\mathcal{L}f \, = \,  (1-x^2)f''-2xf',
$$
whose associated invariant measure $\mu$ is uniform distribution on $[-1,+1]$.
The general family of parameters with the beta distributions as invariant
measures (cf.~\cite[Section~2.7.4]{B-G-L}) may be considered similarly,
at the expense however of tedious computations, as well as multivariate (product) versions.
For simplicity, we only detail this case to better illustrate the conclusion.}

Easy calculations lead to, for a smooth function $f$ on $[-1,+1]$,
\begin{eqnarray*}
\Gamma (f) &=& (1-x^2){f'}^2\\
\Gamma_2(f)&=& (1+x^2){f'}^2 + (1-x^2)^2 {f''}^2 - 2x(1-x^2)f'f''\\
\Gamma_3(f)&=& (1-x^2)^3{f'''}^2 - 6x(1-x^2)^2 f'' f''' - 2(1-x^2)^2f'f''' \\
&&+ 3(1-x^2)(1+3x^2){f''}^2+6x(1-x^2)f'f''+(3-x^2){f'}^2.
\end{eqnarray*}
Observe that
$$
\Gamma_2(f) \, = \, {f'}^2+ \big (xf'-(1-x^2)f'' \big )^2 \, \geq \,  \Gamma (f).
$$
Furthermore,
\begin{equation*} \begin {split}
\Gamma_3(f)-\Gamma_2(f)
    & \, = \, (1-x^2)\Big [
(1-x^2)^2{f'''}^2-6x(1-x^2)f''f'''\\
& \, \quad -2(1-x^2)f'f'''+2(1+5x^2) {f''}^2+8xf'f''+2{f'}^2 \Big ] \\
& \, = \, (1-x^2)\Big [
\big((1-x^2)f'''-3xf''-f'\big)^2+\big(f'+xf''\big)^2+2{f''}^2 \Big ] \, \geq \,  0
\end {split} \end {equation*}
so that $\Gamma_3(f)\geq \Gamma_2(f)$.
Also,
\begin{equation*} \begin {split}
\Gamma_2(f) & \, \geq \, (1+x^2){f'}^2 + (1-x^2)^2{f''}^2 - 2x^2{f'}^2 - \frac12(1-x^2)^2 {f''}^2\\
&=(1-x^2){f'}^2+\frac12(1-x^2)^2 {f''}^2 \\
& \, \geq \, \frac12(1-x^2)^2 {f''}^2.
\end {split} \end {equation*}
{Hence, Theorem \ref{t:abstract} applies with $\rho=\kappa=1$ and $\sigma=\frac12$
(note that $a(x) = 1-x^2$)
to yield the following conclusion. Again, the numerical constants are not sharp.

\begin {proposition} [HSI inequality for the uniform distribution]
Let $\mu$ be uniform probability measure on $[-1,+1]$.
Then, for any $d\nu = h d\mu$ where $h$ is a smooth probability density,
$$
{\rm H} \big (\nu\,|\,\mu \big )
    \, \leq \,   {\rm S}^2 \big (\nu\, |\, \mu \big )
      \, \Psi \bigg ( \frac{{\rm I}(\nu\,|\,\mu)}{2 \, {\rm S^2}(\nu\,|\,\mu)}\bigg)
$$
\end {proposition}
}

\subsection{Families of log-concave distributions}

We consider here a diffusion operator on the line of the type
$$
{\cal L} f \, = \,  f'' - u' f'
$$
associated with a symmetric invariant probability measure
$d \mu = e^{-u} dx$, where $u$ is a smooth potential on $\R$.
The Gaussian model corresponds to the quadratic potential $u(x) = \frac{x^2}{2}$.

We have, for smooth functions $f$,
\begin{eqnarray*}
\Gamma (f) &=& {f'}^2\\
\Gamma_2(f)&=& {f''}^2 + u'' {f'}^2  \\
\Gamma_3(f)&=& {f'''}^2 + 3 u''' f' f'' + 3 u'' {f''}^2 +
   \frac{1}{2} \, \big ( u^{(4)} - u' u ''' + 2 {u''}^2 \big )     {f'}^2 .
\end{eqnarray*}
Assume that there exists $ c >0$ such that, uniformly, $ u'' \geq c $,
\begin {equation} \label {eq4.0}
u^{(4)} - u' u ''' + 2 {u''}^2 - 6c u'' \, \geq \,  0
\end {equation}
and
\begin {equation} \label {eq4.1}
 3 {u'''}^2  \, \leq \,   2(u'' - c )  \big ( u^{(4)} - u' u''' + 2 {u''}^2 - 6c u'' \big ) .
\end {equation}
Then $\Gamma_2(f)\geq c \, \Gamma(f)$, $\Gamma_2(f)\geq {f''}^2$
and $  \Gamma _3 (f) \geq 3c \, \Gamma _2(f)$ for every $f$.
Hence, Theorem~\ref{t:abstract} applies with $\rho=c$, $\kappa=3c$ and $\sigma=1$.

\begin {proposition} [HSI inequality for log-concave distribution]
Let $d\mu = e^{-u} dx$ on $\R$ where $u$ is a smooth potential on $\R$
such that for some $c >0$, $u'' \geq c$ and \eqref {eq4.0} and \eqref {eq4.1} hold.
Then, for any $d\nu = h d\mu$ where $h$ is a smooth probability density,
$$
{\rm H} \big (\nu\,|\,\mu \big)
    \, \leq \,  \frac{1}{2} \, {\rm S}^2 \big (\nu \, | \, \mu \big ) \,
    \Psi \bigg (\frac{{\rm I}(\nu\,|\, \mu)}{c \, {\rm S}^2(\nu\,|\, \mu)} \bigg ).
$$
\end {proposition}

Recall that in this context, the only condition $u'' \geq c > 0$ ensures
the logarithmic Sobolev inequality for $\mu$ \cite[Corollary~5.7.2]{B-G-L}.
It is not difficult to find (simple) examples outside the Gaussian model
(corresponding to $c = \frac{1}{3}$) such that
conditions \eqref {eq4.0} and \eqref {eq4.1} are fulfilled.
For example, if $ u(x) = \frac{x^2}{2} + \varepsilon x^4$,
it is easily seen that these hold for $c = \frac{1}{4} $ and $ \varepsilon = \frac{1}{12}$
(for instance). In the Gaussian case, the estimate obtained
in this proposition is somewhat worse than the HSI inequality of Theorem~\ref {t:hsi}.
At the expenses of more
involved conditions \eqref {eq4.0} and \eqref {eq4.1}, multidimensional versions may
be considered similarly.

\section {Entropy bounds on laws of functionals }\label{S5}

{As emphasized in the introduction, the new HSI inequalities described in the preceding
sections provide entropic bounds on probability measures $\nu$ which may be used towards
convergence in entropy via the Stein discrepancy $S (\nu \, | \, \mu)$.
Now, these bounds assume that the
Fisher information ${\rm I}_\mu (h)$ of the density $h$ of $\nu$ with respect to $\mu$
is finite (in order to control ${\rm I}_\mu (P_t h)$ in small time),
which may or may not hold in specific illustrations.
The goal pursued in the second part of this work is actually to overcome this difficulty
and to describe conditions (integrability and tail behavior) on the initial data itself of a
multidimensional functional $F = (F_1, \ldots, F_d)$ with distribution $\nu = \nu_F$ (on $\R^d$)
in order to control the Fisher information ${\rm I}_\mu (P_t h)$ in small time.
This investigation was initiated in \cite{N-P-S} in Wiener space towards the first normal
approximation results in entropy for Wiener chaos distributions.
Here, we consider distributions of functionals
on a Markov Triple structure $(E, \mu ,\Gamma)$ already put forward
in the preceding section, and describe how the associated $\Gamma$-calculus
may be developed towards normal (as well as gamma) approximations in the entropic sense.

Referring as before to \cite {B-G-L} for a complete account,
we thus deal with a Markov Triple $(E, \mu ,\Gamma)$
on a probability space $(E, {\cal E}, \mu )$, with Markov semigroup ${(P_t)}_{t \geq 0}$
with symmetric and invariant probability measure $\mu$,
infinitesimal generator $\rm L$, associated carr\'e du champ operator $\Gamma$
and underlying algebra of (smooth) functions $\cal A$. Integration by parts expresses that
\begin {equation} \label {e:ipp2}
\int_E f \, {\rm L g} \, d\mu \, = \,  - \int_E \Gamma (f,g) d \mu
\end {equation}
for every $f,g \in {\cal A}$.

The second order differential
operators of Section~\ref {S4} provide instances of this general framework. Gaussian
and Wiener spaces with associated Ornstein-Uhlenbeck semigroup and generator
are a prototypical example for the illustrations. Note in particular that
Wiener chaoses as investigated in \cite {N-P-S}
are eigenfunctions of the Ornstein-Uhlenbeck generator. Eigenfunctions of the
underlying operator $\rm L$ are actually of special interest in the context of the Stein method
as illustrated in Section~\ref {S5.1}.

\medskip

For $d\geq 1$, let $F = (F_1,\ldots ,F_d)$ be a vector defined on $(E, {\cal E}, \mu )$, where each
$F_i$ is centered and square-integrable, and denote by $\nu_F$ the law of $F$.
Common to the three Sections~\ref {S5.1}--\ref {S5.3} below, assume that
the distribution $\nu_F$ of $F$ admits a density $h$ with
respect to the standard Gaussian distribution $\gamma$ on $\R^d$ (in particular,
$\nu_F$ is absolutely continuous with respect to the Lebesgue measure). In the first part,
we describe the Stein kernel and discrepancy for vectors of eigenfunctions of ${\rm L}$.
Next, we address some direct bounds on the Fisher information ${\rm I}_\gamma (h)$
in terms of the data of the functional $F$ and its gradients. Then,
we develop the results on entropic normal approximations, extending
the conclusions in \cite{N-P-S},
by an analysis of the small time behavior of  ${\rm I}_\gamma (P_t h)$.
Finally, we address similar issues in the context of one-dimensional gamma approximation.}

\subsection{Stein kernel and discrepancy for eigenfunctions} \label {S5.1}

The first statement shows that, whenever the vector $F$ is composed
of eigenfunctions of ${\rm L}$, a Stein kernel $\tau_{\nu_F}$
of $\nu_F$ with respect to $\gamma$
as defined in \eqref{e:steinmatrix} can be expressed in terms of
the carr\'e du champ operator $\Gamma$.

\begin{proposition} [Stein kernel for eigenfunctions] \label{p:smg}
Let $F = (F_1,\ldots ,F_d)$ on $(E, {\cal E}, \mu )$ such
that, for every $i=1,\ldots ,d$,
the random variable $F_i$ is an eigenfunction of $-{\rm L}$, with eigenvalue $\lambda_i>0$.
Assume moreover that $\Gamma(F_i,F_j)\in {\rm L}^1(\mu)$ for every $i,j = 1, \ldots, d$.
Then, the matrix-valued map $\tau_{\nu_F}$ defined as
\begin{equation}\label{e:smg}
\tau_{\nu_F} ^{ij}(x_1,\ldots ,x_d) \, = \,  \frac{1}{\lambda_i} \, \E_\mu \Big [\Gamma(F_i,F_j)\, \big|
     \, F = (x_1, \ldots ,x_d)\Big], \quad  i,j=1,\ldots ,d,
\end{equation}
is a Stein kernel for $\nu_F$, that is, it satisfies \eqref{e:steinmatrix}.
(The right-hand side of \eqref{e:smg} indicates a version of the conditional
expectation of $\Gamma(F_i,F_j)$ with respect to $F$ under the probability measure $\mu$.)
\end{proposition}

\begin{proof}
Use integration by parts with respect to ${\rm L}$ to get that,
for every smooth test function $ \varphi  $ on $\R^d$ and every $i = 1, \ldots, d$,
$$
\lambda_i \int_E F_i \,  \varphi (F) d\mu \, = \,  - \int_E {\rm L} F_i \,  \varphi (F) d\mu
     \, = \,  \sum_{j=1}^d \int_E \Gamma (F_i,F_j) \, \frac{\partial \varphi }{\partial x_j} (F)  d\mu.
$$
The proof is concluded by taking conditional expectations.
\end{proof}

As a consequence, together with \eqref{e:s2bound} and Jensen's inequality,
\begin {equation} \label {e:steineigen}
{\rm S}^2 \big (\nu_F\, |\, \gamma \big ) \, \leq \,
 \sum_{i,j=1}^d \frac {1}{\lambda_i^2} \, {\rm Var}_\mu \big (\Gamma(F_i,F_j) \big )
    + \big \|C - {\rm Id} \big \|_{\rm HS}^2 \, = \,  {\rm V}^2
\end {equation}
where $C$ denotes the covariance matrix of $\nu_F$, providing therefore a tractable
way to control the Stein discrepancy in this case.
In addition, combining with the HSI inequality
of Theorem \ref{t:hsi} immediately yields the following statement.

\begin{corollary}
Under the assumptions and notation of Proposition \ref{p:smg},
\begin{equation} \label {eq5.1}
{\rm H} \big (\nu_F\, |\, \gamma \big )
    \, \leq \,  {\rm V^2} \log \bigg (1+ \frac{{\rm I}(\nu_F\, |\, \gamma)}{{\rm V}^2}\bigg ).
\end{equation}
\end{corollary}

{In particular, if $d = 1$ and $C = 1$, $ {\rm H} (\nu_F\, |\, \gamma  ) \to 0$
whenever $ {\rm Var} (\Gamma (F)) \to 0$ (cf.~\cite {N-P-12, L3}). }

%

\begin{example} \label {ex:wiener}
A typical example of a Markov Triple for which the
quantity ${\rm V}^2$ appearing in the above bound can be estimated
explicitly corresponds to the case where $(E, \mathcal{E}, \mu)$ is a
probability  space supporting an isonormal Gaussian process
$X= \{ X(h) : h\in \mathfrak{H}\}$ over some real separable Hilbert
space $\mathfrak{H}$, and ${\rm L}$ is the generator of the associated
Ornstein-Uhlenbeck semigroup. In this case,
$\Gamma(F,G) = {\langle DF, DG \rangle}_\mathfrak{H}$ for smooth functionals $F$ and $G$, where
$D$ stands for the Malliavin derivative operator, and the eigenspaces
of $-{\rm L}$ are the so-called {\em Wiener chaoses} $\{C_k : k\geq 0\}$ of $X$.
For $k = 0,1,2,\ldots$, the eigenvalue of $C_k$ is given by $k$.
A detailed discussion about how to bound a quantity such as ${\rm V}^2$
in the case of random vectors with components inside a Wiener chaos
can be found in \cite[Chapter 6]{N-P-12}. In particular, if $d =1$ and $F$ belongs to $C_k$, then
${\rm V}^2$ can be controlled by the second and fourth moments of $F$ as
\begin {equation*} \begin {split}
{\rm V}^2 \, &= \, \big (\E[F^2]-1\big) ^2 + \frac{1}{k^2} \, {\rm Var}\big ({\|DF\|}_\mathfrak{H}^2 \big) \\
    \, &\leq \,
\big (\E[F^2]-1\big )^2 + \frac{k-1}{3k} \, \big(\E[F^4]-3 \, {\E[F^2]}^2\big).
\end {split} \end {equation*}
In particular, such an estimate provides a proof of the
famous `fourth moment theorem' for chaotic random variables, cf.~\cite[Theorem 5.2.7]{N-P-12}.
\end{example}

\begin {remark} \label {r:linverse}
While eigenfunctions appear as functionals of particular interest
for the control of the Stein discrepancy itself,
the $\Gamma$-calculus actually provides a formal description of Stein kernels
of a given functional $F$ on $(E, \mu, \Gamma)$
(in dimension one for simplicity) as the conditional expectation with respect to $F$
of $ \Gamma (F, {\rm L}^{-1} F) $ (where $ {\rm L}^{-1} F = \int_0^\infty P_t F dt$).
This observation further expands on the preceding example, allowing for
a rather general analysis.
\end {remark}

\subsection{Bounds on the Fisher information} \label {S5.2}

When dealing with the upper-bound \eqref{eq5.1}, the Fisher information
$ {\rm I}(\nu_F\, |\, \gamma) = {\rm I}_\gamma (h)$ of the density $h$
of the law $\nu_F$ of $F$ cannot always be
explicitly deduced from the data concerning the random vector $F$. 
The task of this paragraph is therefore 
to deduce some useful bounds on ${\rm I}(\nu_F\, |\, \gamma)$
in terms of $F$ and its gradients.

\medskip

Let $F = (F_1,\ldots ,F_d)$ be general vector of centered and square-integrable random variables
(that need not necessarily be eigenfunctions of $-{\rm L}$).
Recall that the distribution $\nu_F $ of $F$ is assumed to
admit a (smooth) density $h$ with
respect to the standard Gaussian distribution $\gamma$ on $\R^d$.
It is furthermore implicitly assumed that all the
$F_i$'s are in $\cal A$ (or some extended algebra in the sense of \cite {B-G-L})
allowing for the formal computations developed next. These assumptions should then be verified
on the concrete examples of interest (such as Wiener chaoses).

Let $\phi : \R^d \to \R$ be smooth enough. By integration by parts
\eqref {e:ipp2} with respect to ${\rm L}$,
for every $w \in {\cal A}$, and every $i , j = 1, \ldots, d$,
$$
 \sum_{k =1} ^d \int_E w \, \Gamma (F_i, F_k ) \frac{\partial^2\phi}{\partial x_k\partial x_j} (F) d\mu
  \, = \,    - \int_E {\rm L} F_i \, w \, \frac{\partial\phi}{\partial x_j} (F) d\mu
  -  \int_E \Gamma (F_i,w) \frac{\partial\phi}{\partial x_j}  (F) d\mu .
$$

Let ${\widetilde \Gamma} $ be the symmetric matrix with entries $\Gamma (F_i, F_j)$,
$i, j = 1, \ldots , d$.
Applying the latter to $w = w_{ij}$, symmetric in $i,j$, yields
\begin{equation} \begin {split} \label {eq5.2}
    \int_E {\rm Tr} \big (W  & {\widetilde \Gamma}  \, {\rm Hess} (\phi ) (F) \big )  d\mu \\
    & \, = \,   -  \sum_{i, j=1}^d \int_E {\rm L} F_i \, w_{ij} \, \frac{\partial\phi}{\partial x_j}  (F) d\mu
          -  \sum_{i, j=1}^d \int_E \Gamma (F_i,w_{ij}) \, \frac{\partial\phi}{\partial x_j} (F) d\mu  \\
\end {split} \end {equation}
where $ W = {(w_{ij})}_{1 \leq i, j \leq d}$.
Provided it exists, set $W =  {\widetilde \Gamma} ^{-1}$, so that the left-hand side
in the previous identity is just $ \int_E  \Delta \phi (F) d\mu .$
Recalling from \eqref {e:ou} the Ornstein-Uhlenbeck generator
${\cal L}  = \Delta   - x \cdot \nabla $
associated with the standard Gaussian distribution $\gamma$
on $\R^d$, it follows that
\begin{equation*} \begin {split}
    - \int_E {\cal L} \phi  (F) d\mu
       & \, = \,  \sum_{i, j=1}^d \int_E {\rm L} F_i \, {({\widetilde \Gamma} ^{-1})}_{ij} \,
                   \frac{\partial\phi}{\partial x_j}  (F) d\mu \\
       & \, \quad +  \sum_{i, j=1}^d \int_E \Gamma \big (F_i, {({\widetilde \Gamma} ^{-1})}_{ij} \big)
          \frac{\partial\phi}{\partial x_j}  (F) d\mu
          + \sum_{i=1}^d \int_E F_i \, \frac{\partial\phi}{\partial x_j}  (F) d\mu . \\
\end {split} \end {equation*}
In more compact notation, if
$$
V \, = \,  \bigg ( \sum_{i = 1}^d \Gamma \big (F_i, {({\widetilde \Gamma} ^{-1})}_{ij} \big)
        \bigg)_{1 \leq j \leq d} \qquad {\hbox {and}} \qquad
        U \, = \,  {\widetilde \Gamma} ^{-1} {\rm L} F + V + F ,
$$
then
$$
 - \int_E {\cal L} \phi  (F) d\mu \, = \,  \int_E U \cdot \nabla \phi (F) d\mu .
$$
Applied to $\phi = v = \log h$, by the Cauchy-Schwarz inequality and \eqref {e:fisher},
$$
{\rm I}_\gamma  (h) \, \leq \,  \int_E |U|^2 d\mu .
$$

The consequences of the previous computations are gathered together in the next
statement, where we point out a set of sufficient conditions
{on $F$ and its gradients $\Gamma (F_i, F_j)$} ensuring that the
random variable $U$ is indeed square-integrable.

\begin {proposition} [Bound on the Fisher information] \label{p:direct}
Let $F = (F_1,\ldots ,F_d)$ be a vector of  elements of $\cal A$
on $(E, \mu, \Gamma)$. Assume that all the $F_i$, ${\rm L} F_i $,
$ \Gamma (F_i, F_j)$, $ i, j = 1, \ldots,  d$,
and $ \frac{1}{{\rm det} ({\widetilde \Gamma })}$
are in ${\rm L}^p(\mu )$ for every $ p \geq 1$. Then, $ \int_E |U|^2 d\mu < \infty $ and
\begin {equation} \label {eq5.4}
 {\rm I} \big ( \nu_F \, | \, \gamma \big )  \, \leq \, \int_E |U|^2 d\mu .
\end {equation}
\end {proposition}

{The condition on
$ \frac{1}{{\rm det} ({\widetilde \Gamma })}$ in Proposition~\ref {p:direct}
has some similarity with basic assumptions in Malliavin calculus
(cf.~\cite {N, N-P-12}).}

\begin{example}
One may of course wonder whether the bound \eqref {eq5.4}
is of any interest. Here is a simple example showing that there are instances where
${\rm I}_\gamma (h)$ might be quite intricate to handle directly on the density $h$ of the
distribution of $F$ while $U$ has a clearly description.
On $ E = \R^{2n}$ with the standard Gaussian measure $\mu =\gamma $ and
$\Gamma (f) = |\nabla f|^2$ the standard carr\'e du champ operator, let
$$
F (x) \, = \,  x_1 x_2 + x_3 x_4 + \cdots + x_{2n-1} x_{2n},
    \quad x = (x_1, \ldots, x_{2n}) \in \R^{2n} .
$$
It is classical that the distribution of the product of two independent standard
normal has a density (with respect to the Lebesgue measure on $\R$) given by a
Bessel function. The density $h$ of the distribution of $F$ is thus rather involved.
On the other hand, it is easily seen that
$$
{\rm L} F \, = \,  - 2 F\quad {\hbox {and}} \quad \Gamma (F) \, = \,
 x_1^2 + \cdots + x_{2n}^2 \, = \,  R^2
$$
so that
$$
U \, = \,  F \Big ( - \frac{2}{R^2} - \frac{4}{R ^4} + 1 \Big ) .
$$
By using polar coordinates, it is immediately seen that $\int_ {\R^{2n}} U^2 d\mu < \infty$ as
soon as $n \geq 5$.
\end{example}

\subsection{Fisher information growth and normal approximation} \label {S5.3}

One evident drawback of Proposition~\ref {p:direct} of the previous paragraph is that,
since the quantity $|U|$ is singular as the determinant of ${\widetilde \Gamma }$
is close to $0$, one is forced to assume that $ \frac{1}{{\rm det} ({\widetilde \Gamma })}$
is in all ${\rm L}^p(\mu)$ spaces (or at least for some $p$ large enough depending on $d$).
This assumption is in general too strong, and very difficult to check in concrete situations.
The idea developed in this section (which generalizes the approach initiated in \cite{N-P-S})
is that, under weaker moment assumptions, while the Fisher
information ${\rm I}_\gamma (h)$ might be infinite, it is nevertheless possible to control
the growth as $t \to 0$ of ${\rm I}_\gamma (P_t h)$.
{Together with the control in terms of the Stein discrepancy for large time
achieved in Section~\ref {S2}, one may then reach entropic bounds which can be handled
in concrete examples (such as those of random vectors whose components belong to
some Wiener chaos).}

\medskip

As before, let $F = (F_1,\ldots ,F_d)$ be a general vector of centered
and square-integrable random variables {(in the algebra $\cal A$ or some
natural extension),} with distribution $d\nu_F = hd\gamma$.
As a crucial assumption, $\nu_F$ has a Stein
kernel $ \tau_{\nu_F} $ with respect to $\gamma$ as defined in \eqref{e:steinmatrix}
(see also Proposition \ref{p:smg} and Remark~\eqref {r:linverse}). Recall
the matrix ${\widetilde \Gamma} $ with entries $\Gamma (F_i, F_j)$, $i, j = 1, \ldots, d$.
Also, in what follows we use the convention that, if $\widetilde \Gamma$ is singular,
then the matrix ${\rm det} ({\widetilde \Gamma }) \,
{\widetilde \Gamma }^{-1}$ must be understood as
the transpose of usual adjugate matrix operator
of $\widetilde \Gamma$ (both quantities being of course equal for non-singular matrices).

With the notation of the preceding section, given $\varepsilon  > 0$,
write first, again for a smooth function $\phi $ on $\R^d$ and $\cal L$ the Ornstein-Uhlenbeck
operator in $\R^d$,
\begin{equation*} \begin {split}
\int_E {\cal L} \phi  (F) d\mu
      & \, = \,   \int_E  \Delta \phi  (F) d\mu   - \int_E  F \cdot \nabla \phi (F) d\mu \\
      & \,= \,   \int_E  \frac{{\rm det} ({\widetilde \Gamma })}{{\rm det}
          ({\widetilde \Gamma }) + \varepsilon } \, \Delta \phi  (F) d\mu
         + \int_E  \frac{\varepsilon}{{\rm det} ({\widetilde \Gamma })+ \varepsilon }
              \, \Delta \phi  (F) d\mu \\
       &  \, \quad  - \int_E  F \cdot \nabla \phi (F) d\mu . \\
\end {split} \end {equation*}
Choose
$ W = \frac{ {\rm det} ({\widetilde \Gamma }) \,
{\widetilde \Gamma }^{-1}}{{\rm det} ({\widetilde \Gamma })+ \varepsilon}$ in \eqref {eq5.2},
so that
$$
\int_E   \frac{{\rm det} ({\widetilde \Gamma })}{{\rm det} ({\widetilde \Gamma })
     + \varepsilon } \, \Delta \phi  (F) d\mu
   \, = \,  - \int_E \bigg (\frac{{\rm det} ({\widetilde \Gamma }) \,  {\widetilde \Gamma} ^{-1} {\rm L} F + V_1}  {{\rm det} ({\widetilde \Gamma }) + \varepsilon}
      -   \frac{V_2}{({\rm det} ({\widetilde \Gamma }) + \varepsilon)^2} \bigg )
      \cdot \nabla \phi (F) d\mu
  $$
where
$$
V_1 \, = \,  \bigg ( \sum_{i = 1}^d \Gamma \big (F_i, {\rm det} ({\widetilde \Gamma })
    {({\widetilde \Gamma} ^{-1})}_{ij} \big) \bigg)_{1 \leq j \leq d}
$$
and
$$
V_2  \, = \,  \bigg ( \sum_{i = 1}^d  {\rm det} ({\widetilde \Gamma }) {({\widetilde \Gamma} ^{-1})}_{ij}
    \, \Gamma \big (F_i,  {\rm det} ({\widetilde \Gamma })    \big)
        \bigg)_{1 \leq j \leq d}
$$

Apply now the preceding to $ \phi = P_t v_t$, $v_t = \log P_t h$, $t >0$.
Since $ \nabla P_tv_t (F)  =  e^{-t} P_t (  \nabla v_t) $ and
$$
{\rm I}_\gamma (P_t h) \, = \,  \int_E P_t \big ( |\nabla v_t|^2 \big ) (F) d\mu ,
$$
by the Cauchy-Schwarz
inequality, assuming for simplicity that $0 < \varepsilon  \leq 1$,
\begin {equation*} \begin {split}
\bigg | \int_E   \frac{{\rm det} ({\widetilde \Gamma })}{{\rm det} ({\widetilde \Gamma })
     + \varepsilon } \, & \Delta P_t v_t  (F) d\mu  - \int_E   F \cdot \nabla P_t v_t (F) d\mu \bigg | \\
     & \, \leq \, \frac{ e^{-t}}{\varepsilon ^2}
     \bigg ( \int_E  \Big [ \big | {\rm det} ({\widetilde \Gamma }) \,
       {\widetilde \Gamma} ^{-1} {\rm L} F \big |+ |V_1| + |V_2| + |F| \Big ]^2   d\mu \bigg)^{1/2}
           {\rm I}_\gamma (P_t h)^{1/2} \\
\end {split}Ê\end {equation*}
On the other hand, using the same semigroup computations as in Section~2,
$$
\Delta P_t v_t (F) \, = \,  \frac{e^{-2t}}{\sqrt {1 - e^{-2t}}}
    \int_{\R^d} y \cdot \nabla  v_t \big ( e^{-t} F +  \sqrt {1 - e^{-2t}} \, y \big ) d\gamma (y)
$$
so that
$$
\bigg | \int_E  \frac{\varepsilon}{{\rm det} ({\widetilde \Gamma })+ \varepsilon }
              \, \Delta P_t v_t  (F) d\mu \bigg |
         \, \leq \, \frac{{\sqrt d } \, \varepsilon \, e^{-2t}}{\sqrt {1 - e^{-2t}}}
       \bigg ( \int_E \frac{1}{({\rm det} ({\widetilde \Gamma }) + \varepsilon)^2 }
          \, d\mu \bigg )^{1/2}   {\rm I}_\gamma (P_t h)^{1/2} .
$$

Assume now that
\begin {equation} \label {e:af}
\int_E  \Big [ \big | {\rm det} ({\widetilde \Gamma }) \,
       {\widetilde \Gamma} ^{-1} {\rm L} F \big |+ |V_1| + |V_2|  + |F| \Big ]^2   d\mu
          \, = \,  A_F < \infty
\end {equation}
and that
\begin {equation} \label {eq5.7}
 \int_E \frac{1}{({\rm det} ({\widetilde \Gamma }) + \varepsilon)^2 } \, d\mu
    \, \leq \,  \delta (\varepsilon ) .
\end {equation}
Collecting the preceding bounds and recalling from \eqref {e:information} that
$$
\int_E {\cal L} P_t v_t  (F) d\mu \, = \,  \int_E \mathcal{L}P_tv_t \, hd\mu
  \, = \, - \, {\rm I}_\gamma (P_t h)
$$
yields that, for $t>0$ and $0 < \varepsilon \leq 1$,
\begin {equation} \label {eq5.8}
{\rm I}_\gamma (P_t h) \, \leq \,  2e^{-2t}
  \bigg ( \frac{ A_F}{\varepsilon^4 }
     + \frac{d  \, \varepsilon^2 \delta (\varepsilon )}{1 - e^{-2t}}  \bigg ) .
\end {equation}

\medskip

In the following statement, we determine a handy set of sufficient conditions
on $F$ and its gradients ensuring that,
for some choice of $\varepsilon = \varepsilon (t) >0$, the function
on the right-hand side of \eqref {eq5.8} is integrable for the small values
of $t>0$. Combined with \eqref {eq1.8} for the large values of $t >0$, a control
of the entropy of $\nu_F$ in terms of the Stein discrepancy
$ {\rm S}(\nu_F \, | \, \gamma)$ may then be produced.
Recall the function $\Psi $ on $\R_+$ given by
$\Psi (r) = 1 + \log r $ if $r \geq 1$ and $ \Psi (r) = r $ if $0 \leq r \leq 1$.

\begin {theorem} [normal entropic approximation via Stein discrepancy] \label{t:ggg}
Let $F = (F_1,\ldots ,F_d)$ be a vector of centered elements of $\cal A$
on $(E, \mu, \Gamma)$. Assume that all the $F_i$, ${\rm L} F_i $,
$ \Gamma (F_i, F_j)$, $ i, j = 1, \ldots, d$, are in ${\rm L}^p(\mu )$ for every $ p \geq 1$, and that
\begin {equation} \label {e:bf}
 B_F \, = \, \int_E  \frac{1}{{\rm det} ({\widetilde \Gamma })^\alpha } \, d\mu  < \infty
\end {equation}
for some $\alpha > 0$. Then, $A_F<\infty$ (as defined in \eqref{e:af}) and
\begin {equation} \label {e:ggg}
{\rm H} \big ( \nu_F \, | \, \gamma \big)
      \,  \leq \,   \frac{ {\rm S}^2  (\nu_F \, | \, \gamma  )}{2(1- 4 \kappa)} \,
        \Psi \bigg ( \frac{  2 ( A_F +   d  (B_F+1)) }{ {\rm S}^2(\nu_F \, | \, \gamma)}\bigg)
\end {equation}
where $\kappa = \frac{2 + \alpha}{2( 4 + 3\alpha )} \, (< \frac{1}{4} \, )$.
In particular, under the assumptions on $F$,
$ {\rm H} ( \nu_F \, | \, \gamma ) \to 0$ as ${\rm S} ( \nu_F \, | \, \gamma ) \to 0$.
\end {theorem}

\begin{proof}
First of all, we have that the parameter $A_F$ is finite, since the expressions 
${\rm det} ({\widetilde \Gamma }){\widetilde \Gamma} ^{-1} {\rm L} F $, $V_1$
and $V_2$ only involve
products of $F_i$, ${\rm L} F_i $ and $ \Gamma (F_i, F_j)$, $ i, j= 1, \ldots,  d$.

Now, for every $\varepsilon >0$ and $r >0$,
\begin {equation} \label {e:alpha}
\int_E \frac{1}{({\rm det} ({\widetilde \Gamma }) + \varepsilon )^2} \, d\mu
    \, \leq \,  \frac{1}{\varepsilon ^2} 
    \, \mu  \big (  {\rm det} ({\widetilde \Gamma }) \leq r  \big)
          + \frac{1}{r ^2}
      \, \leq \,  \frac{B_F r ^\alpha}{\varepsilon ^2} + \frac{1}{r ^2} \, .
\end {equation}
The choice of $ r = \varepsilon ^{\frac{2}{\alpha + 2}}$
yields \eqref {eq5.7} with $ \delta (\varepsilon ) = (B_F + 1)
 \varepsilon ^{- \frac {4}{ 2 +\alpha }}$.
Let then
$ \varepsilon = \varepsilon (t) = (1 - e^{-2t})^\kappa $, $t \geq 0$, for
$\kappa = \frac{2 + \alpha}{2( 4 + 3\alpha )}$ ($ < \frac{1}{4}$). Then
$$
\frac{ A_F}{\varepsilon^4 }
     + \frac{d \varepsilon ^2 \delta (\varepsilon )}{1 - e^{-2t}}
        \, \leq \,  \frac{ A_F + d (B_F+1)}{(1 - e^{-2t})^ {4\kappa} }
$$
from which, as a consequence of \eqref {eq5.8}, for every $ t >0$,
\begin {equation} \label {e:fishercontrol}
{\rm I}_\gamma (P_t h)  \, \leq \,  2 \big [ A_F +  d  (B_F+1)   \big] \,
   \frac{ e^{-2t}}{(1 - e^{-2t})^ {4 \kappa} } \, .
\end {equation}

To conclude, recall, {as in the proof of Theorem~\ref {t:hsi},} the decomposition
for every $u >0$,
$$
{\rm H} \big ( \nu_F \, | \, \gamma \big)
     \,  \leq \,    \int_0^u {\rm I}_\gamma  (P_t h)   dt
        + {\rm S}^2 \big ( \nu_F \, | \, \gamma \big)  \int_u^\infty \frac{e^{-4t}}{1 - e^{-2t} } \, dt .
$$
Therefore, by \eqref {e:fishercontrol},
\begin{equation*} \begin {split} \label{e:ee}
{\rm H} \big ( \nu_F \, | \, \gamma \big)
& \, \leq \,   \frac{  A_F +  d  (B_F+1) }{ 1 - 4 \kappa }
      \, (1  - e^{-2u} )^{ 1 - 4 \kappa } \\
  & \, \quad     + \frac{1}{2 } \, {\rm S}^2 \big ( \nu_F \, | \, \gamma \big)
      \big ( - e^{-2u} - \log (1 - e^{-2u}) \big)\\
      & \, \leq \,   \frac{  A_F +  d  (B_F+1) }{ 1 - 4 \kappa }
      \, (1  - e^{-2u} )^{ 1 - 4 \kappa } - \frac{1}{2 } \, {\rm S}^2 \big ( \nu_F \, | \, \gamma \big)
      \log (1 - e^{-2u}) ,
\end {split} \end {equation*}
and the bound \eqref {e:ggg} in the statement follows by optimizing in $ u > 0$
(set $ (1 - e^{-2u})^{1 - 4 \kappa} = r \in (0,1)$.)
Theorem~\ref {t:ggg} is established.
\end{proof}

%
%

Since $\Psi (r ) \leq r$ for every $ r \in \R_+$, observe from \eqref {e:ggg} that
$$
 {\rm H} \big ( \nu_F \, | \, \gamma \big)  \,  \leq \,  \frac{ A_F +   d  (B_F +1) }{ (1 - 4 \kappa )}
$$
so that, under the assumptions of Theorem \ref{t:ggg},
one also has that ${\rm H} ( \nu_F \, | \, \gamma )  < \infty$, a conclusion of independent interest.

The quantity $A_F$ of \eqref {e:af} involves integrability conditions on $F$ and its gradients
(they may actually be weakened according to the precise expression of $A_F$).
On the other hand, $B_F$ of \eqref {e:bf} is rather concerned with a small ball behavior.
For a vector $F = (F_1, \ldots, F_d)$ of eigenvectors of the underlying Markov generator ${\rm L}$,
Theorem \ref{t:ggg} may be combined with \eqref {e:steineigen} to fully control the
relative entropy in terms of $F$ and its gradients as now illustrated in some instances.

\begin{example}
We describe, in part following \cite{N-P-S}, how the preceding developments
may be applied to concrete examples of interest.

\begin{itemize}
\item[(a)] As already mentioned in Example~\ref {ex:wiener},
one such model is the case of a Gaussian vector chaos $ F = (F_1, \ldots , F_d)$,
each $F_i$ being a chaos on Wiener space, in which case (see~Example~\ref {ex:wiener})
$\Gamma(F_i,F_j) = {\langle DF_i, DF_j \rangle}_\mathfrak{H}$.
As put forward in [N-P-S],
the first part of the hypotheses in Theorem~\ref {t:ggg} is fulfilled by the integrability of
Wiener chaoses and of their derivatives.
Concerning the second part of the hypotheses, the relevant property emphasized
in [N-P-S] is that whenever
the law of $F$ has a density with respect to the Lebesgue measure,
which amounts to the fact that $ \E ( {\rm det} ({\widetilde \Gamma })  ) >0$,
then for some universal constant $c >0$,
\begin {equation} \label {eq5.9}
\P \big ( {\rm det} ({\widetilde \Gamma }) \leq r  \big )
   \, \leq \,  c N r ^{1/N} \E \big ( {\rm det} ({\widetilde \Gamma })  \big )^{-1/N}
\end {equation}
for every $r >0$,
where $N \geq 1$ is an integer related to the degrees of the $F_i$'s.
Under \eqref {eq5.9}, the second hypothesis of Theorem~\ref{t:ggg}
clearly holds for any $ \alpha < \frac{1}{N}$ {(cf.~\eqref {e:alpha}).}
The latter then applies to basically recover the main conclusion of [N-P-S].

\item[(b)] It may be observed that the same conclusion \eqref {eq5.9}
holds true when the $F_i$'s
are polynomials under a log-concave measure $d\mu = e^{-u} dx$ on $\R^n$,
at least when $u$ is a polynomial or such that $|\nabla u| \in {\rm L}^p(\mu )$
for every $ p \geq 1$. Indeed, the
determinant $ {\rm det} ({\widetilde \Gamma }) $ is then also of this form, and
the seminal result from [C-W] used in [N-P-S] applies similarly.
This observation allows for an extension of the conclusions of Theorem~\ref{t:ggg}
far away the Gaussian framework.

\end{itemize}
\end{example}

\subsection{Fisher information growth and gamma approximation} \label {S5.4}

This final section develops the analogous investigation towards
gamma approximation, for simplicity one-dimensional.
Denote by $\gamma_p$ the gamma distribution
(on the positive real line) with parameter $p >0$, {invariant measure of the
Laguerre operator}
\begin {equation} \label {e:laguerre}
\mathcal{L}_p f \, = \,  xf''+(p-x)f'.
\end {equation}
Consider a random variable $F\geq 0$
with law $d\nu_F = hd\gamma_p$ absolutely continuous with respect to $\gamma_p$.
Assume that $\nu_F$ admits a Stein kernel $\tau_{\nu_F}$ with respect to $\gamma_p$,
that is, according to \eqref{e:steinmatrix2}  (taking into account the diffusion
coefficient $a(x) = x$ in \eqref {e:laguerre}),
$\tau_{\nu_F}$ is a mapping on $\R_+$ verifying
$$
\int_{\R_+} (x-p) \, \varphi \, d\nu_F  \, = \,  \int_{\R_+} \tau_{\nu_F} \, \varphi ' \, d\nu_F
$$
for every smooth test function $\varphi$.
In particular, $\int_E F d\mu = p$. Note that, in this case,
$$
{\rm S}^2 \big (\nu_F\, |\, \gamma_p \big )
     \, = \,  \int_E \bigg(\frac{\tau_{\nu_F} (F)}{F} - 1\bigg)^2 d\mu.
$$

From the study of Gaussian chaoses for example, and as already mentioned
earlier, it appears
that the latter $ {\rm S} (\nu_F\, |\, \gamma_p  )$ might not always be the
relevant quantity of interest
(cf.~\cite {N-P-09, R}). Indeed, for an eigenfunction $F $ with eigenvalue $ - \lambda $,
$\lambda > 0$, the Stein kernel $\tau_\nu (F)$ may be identified with
the conditional expectation of $ \lambda ^{-1} \Gamma (F)$ knowing $F$.
Now, for such a functional, moment conditions on $F$ may be used to rather control
the variance of $  \lambda ^{-1} \Gamma (F) - F $, and similarly higher
moments (cf.~\cite {A-C-P, A-M-P,L3}).
Of course, by H\"older's inequality,
$$
\bigg ( \int_E \bigg(\frac{\Gamma (F)}{\lambda F} - 1\bigg)^2 d\mu \bigg )^{1/2}
   \, \leq \,  \bigg ( \int_E F^{-2r} d\mu \bigg) ^{1/r}
   \bigg ( \int_F  \bigg | \frac {\Gamma (F)}{\lambda} - F \bigg |^{2s} d\mu \bigg)^{1/s}
$$
for $ r>1$, $\frac{1}{r} + \frac{1}{s} = 1$. Provided it may be ensured that
$ \int_E F^{-2r} d\mu  < \infty$ for some $r >1$, the results here are
nevertheless still of interest.

We assume below that $p\geq \frac12$ so that
the estimates \eqref{abc} and \eqref{e:i1} are verified, with the choice of parameters
$d=1$ and $\rho=\kappa=\sigma=\frac12$ (see the comment preceding
Proposition~\ref {p:gamma}). {The proof of
the following statement will follow the one developed for Theorem~\ref {t:ggg}.}

\begin{theorem} [Gamma entropic approximation via Stein discrepancy] \label{t:lagamma}
On $(E, \mu, \Gamma)$, let $F \geq 0$ in $\cal A$. Assume that
$F$, ${\rm L} F$, $\Gamma (F)$ and $\Gamma (F, \Gamma (F))$ are in
${\rm L}^q(\mu)$ for every $q \geq 1$ and that
$$
   B_F \, = \,  \int_E \frac {1}{\Gamma (F)^\alpha} \, d\mu < \infty
$$
for some $\alpha >0$. Then
$$
A_F \, = \, \int_E \frac1F \Big [ F|{\rm L}F|+\Gamma(F)+
        F \big |\Gamma \big (F,\Gamma(F)\big ) \big |+p+F \Big ] ^2d\mu<\infty
$$
and
$$
{\rm H} \big ( \nu_F \, | \, \gamma_p \big)
       \, \leq \,  \frac{{\rm S}^2(\nu_F \, | \, \gamma_p)}{2 (1 - 4 \kappa )} \,
        \Psi \bigg ( \frac{  2 ( A_F +   B_F+1) }
            { {\rm S}^2(\nu_F \, | \, \gamma_p)}\bigg)
$$
where $\kappa = \frac{2 + \alpha}{2( 4 + 3\alpha )} \, (< \frac{1}{4} \, )$.
In particular, under the assumptions on $F$,
$ {\rm H} ( \nu_F \, | \, \gamma_p ) \to 0$ as ${\rm S} ( \nu_F \, | \, \gamma_p ) \to 0$.
\end{theorem}

\begin{proof}
Denoting by ${(P_t)}_{t\geq0}$ the semigroup with infinitesimal
generator $\mathcal{L}_p$, we have {as in \eqref {e:information},}
$$
{\rm I}_{\gamma_p}(P_th) \, = \,  -\int_{\R_+} \mathcal{L}_p P_tv_t \,hd\gamma_p \, = \,
   -\int_E \mathcal{L}_p P_tv_t(F)d\mu
$$
where $v_t=\log P_th$. Now, for every $\varepsilon>0$,
\begin{equation*} \begin {split}
\int_E \mathcal{L}_p P_tv_t(F)d\mu
     &  \, = \, \int_E F (P_tv_t)''(F)d\mu+\int_E (p-F)(P_tv_t)'(F)d\mu\\
     & \, = \, \int_E  F(P_tv_t)''(F) \, \frac{\Gamma (F)}{\Gamma(F)+\varepsilon} \, d\mu
        +\int_E  F(P_tv_t)''(F) \, \frac{\varepsilon}{\Gamma(F)+\varepsilon} \, d\mu \\
      & \, \quad +\int_E (p-F)(P_tv_t)'(F)d\gamma_p . \\
\end {split} \end {equation*}
By integration by parts,
$$
\int_E  F(P_tv_t)''(F) \, \frac{\Gamma(F)}{\Gamma(F)+\varepsilon} \, d\mu
   \, = \,  \int_E (P_tv_t)'(F)\bigg [ \frac{F(-{\rm L}F)}{\Gamma(F)+\varepsilon}
              - \Gamma \bigg (F,\frac{F}{\Gamma(F)+\varepsilon} \bigg) \bigg ] d\mu .
$$
Using that
$$
\Gamma \bigg (F,\frac{F}{\varepsilon+\Gamma(F)}\bigg )
    \, = \,  \frac{\Gamma(F)}{\Gamma(F)+\varepsilon}-
      \frac{F\,\Gamma(F,\Gamma(F))}{(\Gamma(F)+\varepsilon)^2} \,
$$
it follows that
$$
\int_E \mathcal{L}_p P_tv_t(F)d\mu
    \, = \,   \int_E \sqrt{F}(P_tv_t)'(F)W_\varepsilon(F) d\mu
      + \int_E  F(P_tv_t)''(F) \, \frac{\varepsilon}{\Gamma(F)+\varepsilon} \, d\mu
$$
with
$$
W_\varepsilon(F) \, = \,  \frac{\sqrt{F}(-{\rm L}F)}{\Gamma(F)+\varepsilon}
    -\frac{\Gamma(F)}{\sqrt{F}(\Gamma(F)+\varepsilon)}+
    \sqrt{F} \, \frac{\Gamma(F,\Gamma(F))}{(\Gamma(F)+\varepsilon)^2}
     +\frac{p}{\sqrt{F}}-\sqrt{F}.
$$
Now, for every $0<\varepsilon\leq 1$,
$$
 \big |W_\varepsilon(F) \big |
     \, \leq \,  \frac{1}{\varepsilon^2\sqrt{F}}
      \Big [ F|{\rm L}F|+\Gamma(F)+F \big |\Gamma \big (F,\Gamma(F)\big ) \big |+p+F\Big ].
$$
As a consequence, with the notation introduced in the statement,
$$
\int_E W_\varepsilon^2(F)d\mu \, \leq \,  \frac{A_F}{\varepsilon^4} \, .
$$

{By the Cauchy-Schwarz inequality,}
\begin{equation*} \begin {split}
\bigg|\int_E \sqrt{F}(P_tv_t)'(F)W_\varepsilon(F)d\gamma_p\bigg|
   & \, \leq \,  \bigg (\int_E W_\varepsilon ^2 (F)d\mu \bigg)^{1/2}
  \bigg ( \int_E F{(P_tv_t)'}^2(F)d\mu \bigg)^{1/2} \\
    & \, = \, \bigg ( \int_E W_\varepsilon^2(F)d\mu \bigg)^{1/2}
     \bigg ( \int_{\R_+} \Gamma(P_tv_t)hd\gamma_p \bigg)^{1/2}.
\end {split} \end {equation*}
{Since $ \Gamma(P_tv_t) \leq  e^{-t} P_t( \Gamma (v_t))$ (Theorem 3.2.4 in [B-G-L]),}
\begin{equation*} \begin {split}
\bigg|\int_E \sqrt{F}(P_tv_t)'(F)W_\varepsilon(F)d\gamma_p\bigg|
& \, \leq \,  e^{-t/2}\bigg( \int_E W_\varepsilon^2 (F)d\mu \bigg)^{1/2}
   \bigg ( \int_{\R_+} P_t \big ( \Gamma(v_t) \big ) h d\gamma_p \bigg)^{1/2} \\
& \, \leq \,  \frac{e^{-t/2}}{\varepsilon^2}\, A_F^{1/2} \, {\rm I}_{\gamma_p}(P_th)^{1/2}.
\end {split} \end {equation*}

On the other hand, the estimate \eqref{abc} yields the bound
\begin{equation*} \begin {split}
\int_E F^2(P_tv_t)''(F)^2d\mu
    & \, = \, \int_{\R_+} x^2{(P_tv_t)''}^2 hd\gamma_p\\
    & \, \leq \,  \frac{1}{e^t-1}\int_{\R_+} P_t \big (\Gamma(v_t) \big ) hd\gamma_p\\
    & \, = \,  \frac{1}{e^t-1}\int_{\R_+} \Gamma(v_t)P_thd\gamma_p
      \, = \,   \frac{1}{e^t-1} \, {\rm I}_{\gamma_p}(P_th).
\end {split} \end {equation*}
This in turn implies that
$$
\left|\int_E F(P_tv_t)''(F) \, \frac{\varepsilon}{\Gamma(F)+\varepsilon} \, d\mu\right|
    \, \leq \,  \frac{1}{\sqrt{e^t-1}} \, {\rm I}_{\gamma_p}(P_th)^{1/2}
       \bigg ( \int_E \left( \frac{\varepsilon}{\Gamma(F)+\varepsilon}\right)^2d\mu \bigg)^{1/2}.
$$

Gathering together all the previous estimates, we deduce that,
for every $0<\varepsilon\leq 1$ and $t>0$,
$$
{\rm I}_{\gamma_p}(P_th)  \, \leq \,
   \frac{2 e^{-t}A_F}{\varepsilon^4} +
 \frac{2}{e^t - 1}\,\int_E \left( \frac{\varepsilon}{\Gamma(F)+\varepsilon}\right)^2d\mu .
$$
{On the basis of this estimate, we then conclude
exactly as in the proof of Theorem~\ref {t:ggg}.}
\end{proof}

\end{document}